\documentclass[11pt,leqno]{amsart}\usepackage{bbm}
\usepackage{mathrsfs}
\usepackage{amssymb}
\usepackage{amsmath}
\usepackage{amsthm}
\usepackage{amsfonts}
\usepackage{color}
\usepackage{graphicx}
\usepackage[active]{srcltx}
\usepackage{CJK}

\usepackage[cp1252]{inputenc}

\usepackage{mathrsfs}
\usepackage{graphicx}

\usepackage[active]{srcltx}

\allowdisplaybreaks

\usepackage{titletoc}
\def\rr{{\mathbb R}}
\def\rn{{{\rr}^n}}

\def\nn{{\mathbb N}}

\def\fz{\infty}
\def\az{\alpha}

\def\dist{{\mathop\mathrm{\,dist\,}}}
\def\loc{{\mathop\mathrm{\,loc\,}}}

\def\lz{\lambda}
\def\dz{\delta}

\def\ez{\epsilon}

\def\bz{\beta}

\def\gz{{\gamma}}

\def\tz{\theta}
\def\sz{\sigma}

\def\wz{\widetilde}

\def\bint{{\ifinner\rlap{\bf\kern.35em--}
\int\else\rlap{\bf\kern.45em--}\int\fi}\ignorespaces}

\def\bbint{{\ifinner\rlap{\bf\kern.35em--}
\hspace{0.078cm}\int\else\rlap{\bf\kern.45em--}\int\fi}\ignorespaces}

\def\esup{\mathop\mathrm{\,esssup\,}}

\newtheorem{thm}{Theorem}[section]
\newtheorem{lem}[thm]{Lemma}%[section]     %@@!!@@!!
\newtheorem{rem}[thm]{Remark}%[section]    %@@!!@@!!
\newtheorem{cor}[thm]{Corollary}%[section]    %@@!!@@!!
\numberwithin{equation}{section}

\begin{document}
%\begin{CJK*}{GBK}{song}

\arraycolsep=1pt

\title[Everywhere differentiability of absolute minimizers]{ Everywhere differentiability of
absolute minimizers \\
for locally strongly convex and concave \\
Hamiltonian $H(p)\in C^0(\rn)$ with $n\ge3$}

\author{ Peng Fa, Qianyun Miao and Yuan Zhou}
\address{ Department of Mathematics, Beihang University, Beijing 100191, P. R. China}
\email{pengfa@buaa.edu.cn }

\address{ School of Mathematical Sciences, Peking University, Beijing 100871, P. R. China}
\email{qianyunm@math.pku.edu.cn}

\address{ Department of Mathematics, Beihang University, Beijing 100191, P. R. China}
\email{yuanzhou@buaa.edu.cn}

%\title{\Large\bf Everywhere differentiability of absolute minimizers
%involving Hamiltonian $H\in C^0(\rr^n)$ when $n\ge3$}
%\author{Peng Fa, Qianyun Miao and Yuan Zhou}
\date{\today}
\maketitle

\begin{center}
\begin{minipage}{15cm}\small
{\noindent{\bf Abstract.}\quad
Suppose that
$n\ge3$ and  $H(p)\in C^0(\rr^n)$ is a locally strongly convex and concave Hamiltonian. We obtain the everywhere differentiability of all absolute minimizers for
$H$ in any domain of $\rn$.
}
\end{minipage}
\end{center}

\section{Introduction}
Let $n\ge2$ and suppose  that $H\in C^0(\rn)$ is  convex and coercive (i.e., $\liminf_{p\to\fz}H(p) =\fz$).
Aronsson   1960's initiated   the study of  minimization problems for   the $L^\infty$-functional  $${\mathcal F}_{H}\left(u,\Omega\right)=\esup_{x\in\Omega}H\left(Du (x) \right)   \  \mbox{for any domain $\Omega\subset\rn$ and function $u\in W^{1,\fz}_\loc\left(\Omega\right)$};$$
 see \cite{a1,a2,a3,a4}.
 Given  any  domain $\Omega\subset\rn$, by Aronsson a function $u\in W^{1,\infty}_{\loc}\left(\Omega\right)$  is  called an absolute minimizer  for $H$ in $\Omega$ (write $u\in AM_H\left(\Omega\right)$ for simplicity)  if
$${\mathcal F}_{H}\left(u,V\right)\le {\mathcal F}_{H}\left(v,V\right)\  \mbox{whenever $V\Subset\Omega$, $v\in W^{1,\infty}_{\loc}\left(V\right)\cap C (\overline V )$ and $u=v$ on $\partial V$}.$$
It turns out that the absolute minimizer  is
  the correct notion of minimizers for such  $L^\infty$-functionals.

 The   existence of absolute minimizers for given continuous boundary in bounded domains
 was proved by Aronsson \cite{a3} for $\frac12|p|^2$
  and Barron-Jensen-Wang \cite{bjw} for general $H(p)\in C^0(\rn)$;
   while their  uniqueness was built up by Jensen
\cite{j93}  for $\frac12|p|^2$ (see also \cite{as,bb,cgw}),  %Crandal-Gunnarsson-Wang \cite {cgw} for Banach norms,
 and by Jensen-Wang-Yu \cite{jwy}  and
 Armstrong-Crandal-Julin-Smart \cite{acjs} for $H(p)\in C^2(\rn)$  and  $H(p)\in C^0(\rn)$, respectively,
 with $H^{-1}(\min H)$ having empty interior.

 Moreover, if $H\in C^1(\rn)$ is  convex and coercive,
  absolute minimizers coincide with viscosity solutions to
   the Aronsson equation (a highly degenerate  nonlinear elliptic equation)
\begin{equation}\label{eq1.2}
\mathscr A_H(u):=\sum_{i,j=1}^n
H_{p_i} \left(Du\right) H_{p_j} \left(Du\right) u_{x_ix_j}   =0\quad\mbox{\rm in}\;\Omega,
\end{equation}
see Jensen \cite{j93}  for $H(p)= \frac12|p|^2$,  and Crandall-Wang-Yu \cite{cwy}   and  Yu  \cite{y06} (and also \cite{acjs, bjw,c03,gwy06,cgw}) in general.
 Here $H_{p_i}=\frac{\partial H} {\partial p_i}$ for $H\in C^1(\rn)$,
 $u_{x_i}=\frac{\partial u} {\partial x_i}$ for $u\in C^1(\rn)$,   and
 $u_{x_ix_j}=\frac{\partial^2 u} {\partial x_i\partial x_j}$  for $u\in C^2(\rn)$.
 For the theory of viscosity solution see \cite{CIL}.
    In the special case $H(p)= \frac12|p|^2$, the Aronsson equation \eqref{eq1.2} is  the $\fz$-Laplace equaiton
   \begin{equation}\label{eq1.2xxx}
\Delta_\fz u:=\sum_{i,j=1}^n
u_{x_i}   u_{x_j}   u_{x_ix_j}     =0\quad\mbox{\rm in}\;\Omega
\end{equation}
  and its viscosity solutions   are called as $\infty$-harmonic functions.
   If $H\in C^0(\rn)$ but  $\not\in C^1(\rn)$, we refer to \cite{cgw,acjs} for further discussions and related problems on the Euler-Lagrange equation for absolute minimizers.

  The regularity of absolute minimizer   is then the main issue in this field.

By Aronsson \cite{a84},  $\fz$-harmonic functions are not necessarily $C^2$-regular;
indeed $\fz$-harmonic function
$x_1^{4/3}-x_2^{4/3}$ in whole $\rn$  is    not $C^2$-regular.
Such a function also leads to a well-known  conjecture on the $C^{1,1/3}$- and $ W^{2,t}_\loc$-regularity  with $1\le t<3/2$ of $\fz$-harmonic functions. A seminar step  towards this is made by   Crandall-Evans  \cite{ce}, who obtained   their linear approximation property. They \cite{ceg} also proved   that all   bounded $\fz$-harmonic functions  in whole $\rn$ with $n\ge2$
 must be  constant functions.

Next, when $n=2$, Savin \cite{s05} established their interior $C^1$-regularity
 and then deduced the corresponding Liouville theorem, that is,
 all  $\fz$-harmonic functions  in whole  plane
  with a linear growth at $\fz$ (that is,  $|u(x)|\le C(1+|x|)$ for all $x\in\rr^2$) must be linear functions.
   Later, the  interior $C^{1,\alpha}$-regularity for some $0<\alpha\ll 1/3$
     was proved by Evans-Savin  \cite{es08} and the   boundary $C^1$-regularity by Wang-Yu \cite{wy12}.
 Recently, Koch-Zhang-Zhou \cite{kzz} proved that $|Du|^\alpha\in W^{1,2}_\loc$ for   all $\alpha>0$ and
all   $\infty$-harmonic functions $u$ in planar domains,
which is sharp as $\alpha\to0$;
  also  that the distributional determinant $-\det D^2u\,dx$ is a nonnegative Radon measure.

 Moreover, when $n\ge 3$, Evans-Smart \cite{es11a,es11b}  obtained their everywhere differentiability;
Miao-Wang-Zhou \cite{mwz2} and Hong-Zhao
\cite{hz} independently observed
an asymptotic Liouville property, that is,  if
 $u$ is a $\fz$-harmonic function   in whole  $\rn$
  with a linear growth at $\fz$, then $\lim_{R\to\fz}\frac1Ru(Rx)=  e\cdot x$ locally uniformly for some vector $e$ with $|e|=\|Du\|_{L^\fz(\rn)}$.
   But $C^1,C^{1,\alpha}$-regularity and the corresponding Liouville theorem of $\infty$-harmonic functions are completely open.

%  Regards of absolute minimizers for general $H$,
%  to recall the known progress in their regularity,

  On the other hand, if $H \in C^2(\rn)$ is locally strongly convex, Wang-Yu \cite{wy} obtained
  the linear approximation property of absolute minimizer, and when $n=2 $,
   the $C^1$-regularity  and hence the corresponding Liouville theorem. In this paper, we say that $H\in C^0(\rr^n)$ is locally strongly convex (resp. concave)
if for any convex subset $U$ of $\rn$, there exists $  \lambda >0$ depending on $U$ (resp. $\Lambda>0$) such that
$$ H(p)-\frac{\lambda}{2}|p|^2 \quad  \mbox{
(resp. }\ \frac{\Lambda |p|^2}{2}-H(p) \mbox{) is convex in $U$}. $$
Note that %if $H\in C^2(\rn)$, then the locally strong  convexity of $H$   is equivalent to that
%$D^2H\ge c_RI_n$ in $B(0,R)$ for all $R>0$ and some $c_R>0$. Moreover,
$H\in C^2(\rn)$ implies that
$H$ is always locally strongly concave.  In particular, the $l_\alpha$-norm for $2<\alpha<\fz$ provides a class of typical example  of   locally strongly convex and concave but non-Hilbertian Hamiltonians.

Recently, under the assumptions that $H\in C^0(\rn)$ is convex and coercive, it was shown by Fa-Wang-Zhou \cite{fwz}   that
   $H$ is not a constant in any line segment
if and only if  all absolute minimizers for $H$ have the linear approximation property; moreover, when
  $n=2$,  if and only if  all absolute minimizers for $H$ are $C^1$-regular, and also
if and only if  the corresponding Liouville theorem holds.
In \cite{fmz},
we proved that  if  $H \in C^2(\rr^2)$ is locally strongly convex and concave, then
$H(Du)^\alpha\in W^{1,2}_\loc$ for   all $\alpha>\frac12-\tau_H$
for all   absolute minimizers $u$ in planar domains, where   $0<\tau_H\le \frac12$ and  $\tau_H=1/2$ when $H\in C^2(\rr^2)$; and also  that the distributional determinant $-\det D^2u\,dx$ is a nonnegative Radon measure.
 But, when $n\ge 3$, the  everywhere differentiability,
  $C^1, C^{1,\alpha}$-regularity and the Liouville theorem is not clear.

If $n\ge3$ and
    $H\in C^0(\rr^n)$ is locally strongly convex and  concave,
this paper aims   to prove the following everywhere differentiability (Theorem \ref{Theorem} below) and asymptotic Liouville property (Theorem \ref{Theorem2} below)
 of absolute minimizers.

 \begin{thm}\label{Theorem}
Suppose that  $n\ge3$ and $H\in C^0(\mathbb{R}^n)$ is locally strongly convex and concave.
Let $\Omega\subset\rn$ be any domain. If $u\in AM_H(\Omega)$,
 then $u$ is  differentiable everywhere in $\Omega$.
\end{thm}

%We also have the following differentiability at $\fz$ (an asymptotic Liouville property).

 \begin{thm}\label{Theorem2}
Suppose that  $n\ge3$ and $H\in C^0(\mathbb{R}^n)$ is locally strongly convex/concave.
If $u\in AM_H(\rn)$ with a linear growth at $\fz$,
then there exists a unique vector $e$ such that
 $$H(e)=\|H(Du)\|_{L^\fz(\rn)}\quad{\rm and }\quad \lim_{R\to\fz} \frac1{R}u(Rx)=e\cdot x\quad\mbox{ locally uniformly in $\rn$}.$$
\end{thm}

%Recall that the  existence and uniqueness  of $\infty$-harmonic functions in bounded domains
%has been established by Jensen \cite{j93}; see also \cite{bb,cgw,as,pssw} for other approaches to  the uniqueness.
%When $H\in C^0(\rn)$ is convex and coercive,
% we refer to Barron et al \cite{bjw} for  the existence
% of absolute minimizers.
% If  $H^{-1}\left(\min  H  \right)$ has empty interior   in addition,
% see Armstrong et al \cite{acjs} (see also  \cite{jwy,acj,cgw}) for  their uniqueness.
%By Yu \cite{y06} (see also \cite{jwy,acjs}), to get the uniqueness it is also necessary to assume $H^{-1}\left(\min  H  \right)$  having empty interior.

%The regularity for absolute minimizers is the main issue in this field. Recall that the existence and uniqueness of absolute minimizers has already been well-understood  by  Jensen
%\cite{j93}  for $\frac12|p|^2$ (see also \cite{bb, as}) and Crandal-Gariepy-Wang \cite {cgw} for Banach norms,
% by Jensen-Wang-Yu \cite{jwy} $H(p)\in C^2(\rn)$ and
% Armstrong-Crandal-Julin-Smart \cite{acjs} for   $H(p)\in C^0(\rn)$.

%By the definition and \cite{j93},    $\fz$-harmonic functions  are  locally Lipschitz,
%and hence differentiable almost everywhere.

%
%When $n\ge3$, by Theorems \ref{Theorem}\&\ref{Theorem2} if $H\in C^0(\mathbb{R}^n)$ is locally strongly convex and concave,
%then absolute minimizers are differentiable everywhere and also the corresponding asymptotic Liouville property holds.
%Thanks to \cite{fwz}, when $n=2$, this is true.

When $n\ge3$, it is unclear to us  whether   the assumption for $H$ in Theorems \ref{Theorem}\&\ref{Theorem2} can be relaxed to  the weaker (and also necessary in some sense) assumption that $H\in C^0(\rn)$ is convex and coercive and  is not a constant in any line segment. By  \cite{fwz}, if $H\in C^0(\rn)$ is convex and coercive,  and is constant in some line-segment, then both of Theorems \ref{Theorem}\&\ref{Theorem2} are not necessarily true.

In particular,  it would be interesting to prove the everywhere differentiability of
 absolute minimizer for   $l_\alpha$-norm with  $1<\alpha<2$.
 Recall that if $2<\alpha<\fz$, then
 $l_\alpha$-norm  belongs to  $C^2(\rn)$ and is convex, and hence both of the conclusions of
 Theorem \ref{Theorem}\&\ref{Theorem2} holds.
  If $\alpha=1 $ or $\fz$,  the $l_\alpha$-norm will be constant in some line-segment.

By Remark \ref{r1.3} below, we only need to prove Theorems \ref{Theorem}\&\ref{Theorem2} when $H\in C^0(\rn)$ satisfies
  \begin{enumerate}
   \item[(H1)]  $H$ is   strongly convex and concave in $\rr^n$, that is, there exist $0<\lambda\le \Lambda<\fz$ such that $$\mbox{ both of $ H(p)-\frac{\lz}2|p|^2$ and  $\frac\Lambda{2 }|p|^2- H(p)$ are convex in $\rn$.}$$
    \item[(H2)]  $H(0)=\min_{p\in \rn}H(p)=0$.
   \end{enumerate}

\begin{rem}\label{r1.3}\rm  Suppose that  $H\in C^0(\mathbb{R}^n)$ is locally strongly convex and concave.

(i)
If   $u\in AM_H(\Omega)$ for some domain $\Omega\subset\rn$,
letting $U\Subset\Omega$ be arbitrary subdomain,
  we have $k  =\| Du \|_{L^\fz(U)}<\fz$.
Next, by \cite[Lemma A.8]{fmz}, there exists a $\wz H$ which is
     strongly convex/concave in $\rr^n$ and $\wz H= H$ in $B(0,k+1)$.
     Thus $u\in   AM_{\wz H}(U)$.
The strongly convexity of $\wz H$ implies that
 there exists a $p_0\in\rn$ such that $\min_{p\in\rn}\wz H(p)=H(p_0)$.
Set  $\bar H(p)= \wz H(p+p_0)-\wz H(p_0) $ for $p\in\rn$. Then $\bar H$
 satisfies (H1)\&(H2).
Write $\bar u(x)=u(x)-p_0\cdot x$ for all $x\in U$.
We have $\bar u  \in AM_{\bar H} (U ) $.
 Since $u$ and $\bar u$ have the same regularity in $U$,
 we only need to prove the everywhere differentiability of $\bar u$ in $U$.

 (ii) If
  $u\in AM_H(\rn)$ has a linear growth at $\fz$, then by \cite{fwz} we have  $k:=\|Du\|_{L^\fz(\rn)}<\fz$.
  Let $\bar u$ and $\bar H$ as above. Then $u$ is linear if and only if $\bar u$ is linear.
  So we only need to prove $\bar u$ is linear.
 \end{rem}

Unless other specifying, we always assume that $H\in C^0(\rn)$ satisfies (H1)\&(H2) below.
Note that the geometric\&variational approach used in dimension 2 (see Savin \cite{s05} and also \cite{fwz,wy}) is not enough to prove Theorems \ref{Theorem}\&\ref{Theorem2},
since it includes a key planar topological argument. Moreover, since
$H\in C^0(\rn)$ does not have Hilbert structure necessarily,
it is not clear whether one can prove Theorem \ref{Theorem} by using the idea of Evans-Smart \cite{es11a}---a PDE approach based on maximal principle (see also Remark \ref{rem} (ii)).
But,  in Section 2, we are able to prove Theorems \ref{Theorem}\&\ref{Theorem2} by borrowing some idea of Evans-Smart \cite{es11b}---a PDE approach based on an adjoint argument, and using  the following crucial ingredients:
  \begin{enumerate}
 \item[(a)] the linear approximation property of any given absolute minimizer $u$ for $H$
 as obtained in Fa-Wang-Zhou \cite{fwz} and Wang-Yu \cite{wy} (see Lemmas 2.1\&2.5).

 \item[(b)] a stability result  in \cite{fmz} (see Lemma 2.2)
 which allows to approximate $u$   via absolute minimizers $u^\gz$ of a Hamiltonian $H^\gz$, where
  $H^\gz$ is a smooth approximation of  $H$  and satisfies (H1)\&(H2) with the same constants $\lz, \Lambda$.

   \item[(c)] a  uniform approximation  to  $u^\gz$  via smooth functions $u^{\gz,\ez}$ (see Theorem \ref{ex-una}),  which is an appropriate modification of Evans' approximation via $e^{\frac1\ez H^\gz}$-harmonic functions in  \cite{e03}.
       The point is that none of $k\ge3$ -order derivatives of $H^\gz$ is involved in
       the linearization of the  equation \eqref{apeq2} for $u^{\gz,\ez}$.

 \item[(d)] an integral flatness estimate for $u^{\gz,\ez}$ (see Theorem \ref{Fla}).
\end{enumerate}

Theorem \ref{ex-una} will be proved in Section 3.
 The novelty in the proof of Theorem \ref{ex-una} is that  we use  viscosity  solutions to certain Hamilton-Jacobi equation as barrier functions to get a boundary regularity of $u^{\gz,\ez}$ and then conclude the uniform approximation of $u^{\gz,\ez}$ to $u^\gz$.
 The reason to use $u^{\gz,\ez}$ instead of $e^{\frac1\ez H^\gz}$-harmonic functions
 is that the linearization of $e^{\frac1\ez H^\gz}$-harmonic equation contains $3$-order derivatives of $H^\gz$; see Remark \ref{rem} (i) for details.

  Theorem \ref{Fla} will be proved in Section 5.
To this end, we generalize in Section 4 the adjoint arguments of \cite{es11b} to Hamiltonian $H^\gz$ and $u^{\gz,\ez}$.
Since none of $k\ge3$ -order derivatives of $H^\gz$ is involved in
       the equation for $u^{\gz,\ez}$, all key estimates in Theorem \ref{ex-una} and Section 4
       rely only on $\lz$ and $\Lambda$. This is indeed important to get Theorem \ref{Fla}.
Moreover, since  $H\in C^0(\rn)$ does not have Hilbert structure in general,
some new ideas are needed to get Theorem \ref{Fla} in Section 5; in particular, the test function used in the proof of flatness estimates in \cite{es11b}  is not enough to us,
as an another novelty we find a suitable test function and build up some related estimates.

\section{Proofs of Theorems \ref{Theorem}\&\ref{Theorem2}}

%Below we always assume that   $H\in C^0(\rr^n)$ satisfies (H1)\&(H2).
%In section we prove Theorem \ref{Theorem}
%   To prove Theorem \ref{Theorem},  we need the linear approximation property established in
%    \cite{fwz}, the approximation property in \cite{fmz} and also a generalization of Evans-Smart's flatness via adjoint methods  \cite{es11a}.
%
%    Theorem \ref{Theorem2} would follow from
%        the linear approximation property at the infinity as established in
%    \cite{fwz} and an argument similar to (indeed, line by line)  that used to prove  Theorem \ref{Theorem}
%
%   \subsection{Proof of Theorem \ref{Theorem}}

Considering Remark \ref{r1.3}, we always assume  that  $H\in C^0(\rn)$ satisfies (H1)\&(H2).
To prove Theorem \ref{Theorem}, let $\Omega$ be any domain of $\rn$, and    $u\in AM_H(\Omega) $.
We recall    the following linear
    approximation property of $u$ as established by   \cite{fwz}.
  \begin{lem} \label{lla} For   any $x\in \Omega$ and any sequence $\{r_j\}_{j\in\nn}$ which converges to $0$,
there exist a subsequence $
  \{r_{j_k}\}_{k\in\nn}$ and a vector $e_{\{r_{j_k}\}_{k\in\nn}}$ such that    $$\lim_{k\to\fz}\sup_{y\in B\left(0,1\right)}\left|\frac{u (x+r_{j_k}y )-u\left(x\right)}{r_{j_k}} -  e_{\{r_{j_k}\}_{k\in\nn}}\cdot y \right|=0 $$
  and $$H (e_{\{r_{j_k}\}_{k\in\nn}}  ) =\lim_{r\to0 } \|H\left(Du\right)\|_{L^\fz\left(B\left(x,r\right)\right)}.$$
 \end{lem}
For each $x\in\Omega,$ denote by $\mathscr D u(x)$   the collection of all possible vector $e_{\{r_{j_k}\}_{k\in\nn}}$ as above.
Observe that  $u$ is      differentiable    at $x $  if and only if
 $\mathscr D u(x)$ is a singleton; in this case $\mathscr D u(x)=\{Du(x)\}$.

To see that  $\mathscr D u(x)$ is a singleton,
   we need the following approximation to $u$   given in \cite{fmz}.
  Precisely, let $\{H^\gamma\}_{\gamma\in(0,1]}$ be a standard smooth
approximation to $H$  as below.
 For each  $\gz\in(0,1]$, let $\wz H^\gz=\eta_\gz\ast H$, where $\eta_\gz$ is standard smooth mollifier. Since $H^\gz$ is strictly convex  there exists a unique point $p^\gz\in \rr^2$ such that
 $\wz H^\gz(p^\gz)= \min_{p\in\rr^2} \wz H^\gz(p) .$
Set
\begin{equation} \label{eq7.1} H^\gz(p)=\wz  H^\gz(p+p^\gz)-\wz H^\gz(p^\gz)\quad\forall p\in\rn.\end{equation}
 Obviously, $H^\gamma$ satisfies (H2); by \cite[Appendix A]{fmz},
 $\{H^\gamma\}_{\gamma\in(0,1]}$ satisfies (H1) with the same $\lz$ and $\Lambda$,  and
$ H^\gamma\to H$  locally uniformly as $\gz\to0$.
For each $\gz\in(0,1]$ and $U\Subset\Omega$, let
$$\mbox{$u^\gz\in AM_{H^\gz}(U)\cap C^0(\overline U)$ with $u^\gz=u$ on $\partial U$. }$$
We then have the  following result; see
 \cite{fmz} for $n=2$  and note that the proofs  in \cite{fmz} also works for $n\ge3$.
\begin{lem}\label{abso}
We have
$$ \|H^\gz(Du^\gz)\|_{L^\fz (  U)}\le C \Lambda \|u^\gz\|_{C^{0,1}(\partial U)}\quad\forall \gz\in(0,1],$$
and  $u^\gz\to u$ in $C^{0}(U)$ as $\gz\to0$.
\end{lem}

Next,
for  any $ \gz\in(0,1]$,  to approximate $u^\gz$ in a smooth way we consider   the following Dirichlet problem:
\begin{equation}\label{apeq2}
\mathscr{A}_{H^{\gamma}} (v)+\ez \Delta v=0\quad {\rm in \ }\
U;\ v=u^\gz\ {\rm on} \ \partial U.
\end{equation}
The following result is proved in Section 3.
\begin{thm}\label{ex-una}
 For each $\ez,\gz\in(0,1]$, there exists a unique solution
  $u^{\gz,\ez} \in C^\fz(U)\cap C^{0}(\overline U)$ to \eqref{apeq2}. Moreover, the following hold.
  \begin{enumerate}
\item[(i)] We have
$$ \|u^{\gz,\ez}\|_{C^0(\overline U)}\le \|u^{\gz}\|_{C^0(\partial U)}\quad\forall  \ez\in(0,1].
$$

\item[(ii)]  We have
$$ \|Du^{\gz,\ez}\|_{C^0(\overline V)}\le C_0(\lambda,\Lambda,{\rm dist}(V,\partial U) ,\|u^{\gz}\|_{C^0(\partial U)} )
  \quad\mbox{ for any $V\Subset U$ and $  \ez\in(0,1]$,}
$$
 where the constant $C_0(\lambda,\Lambda,{\rm dist}(V,\partial U) ,\|u^{\gz}\|_{C^0(\partial U)} )$   depends only on $\lz$, $\Lambda$, ${\rm dist}(V,\partial U)$ and
$\|u^{\gz}\|_{C^0(\partial U)}$.
 \item[(iii)] There exist  $\ez_\ast>0$ and $C_\ast>0$ depending on  $H^\gz$ and $ \|u^\gz\|_{C^{0,1}(\overline U)} $ such that
for any $\ez\in (0,\ez_\ast)$, we have
$$
|u^{\gz,\ez}(x)-u^{\gz}(x_0)|\le C_\ast |x-x_0|%\max\left\{\mathcal{ L}^\lambda_\sz(x_0,x),\mathcal{\wz L}^\lambda_\sz(x_0,x)\right\},
\quad \forall x\in U,\ x_0\in \partial U.
$$
\item[(iv)]  We have $u^{\gz,\ez} \to u^\gz$  in $C^0(\overline U)$ as $\ez\to 0$.
\end{enumerate}
  %
%  $$ \|H(Du^\gz)\|_{L^\fz(V)}\le C_0(\lambda,\Lambda)\frac1{\dist(V,\partial U)} \| u\|_{C^{0 }(\partial U)}\quad\forall \ez,\gz \in(0,1],$$
%  and moreover $u^{\gz,\ez} \to u^\gz$  in $C^0(\overline U)$ as $\ez\to 0$.

  \end{thm}
   %The proof of Theorem \ref{ex-una} is postponed to Section 3.
  The existence and uniqueness of $u^{\gz,\ez}$, and also
  Theorem \ref{ex-una} (i) follow  from the classical elliptic theory;
   Theorem \ref{ex-una} (iv) from Theorem \ref{ex-una} (ii) and (iii).
   Theorem \ref{ex-una} (ii) follows from the approach by \cite{es11a} based on  the maximal principle  and the linearized operator arising from \eqref{apeq2}:
  \begin{equation}\label{lin-op1}
 -
H_{p_i}(Du^{\gz,\ez})H_{p_j}(Du^{\gz,\ez})v_{x_ix_j}-2H_{p_ip_l}(Du^{\gz,\ez})u^{\gz,\ez}_{x_ix_j}H_{p_j}(Du^{\gz,\ez})
v_{x_l}-\ez \Delta v.
\end{equation}
  Since   none of $k\ge 3$ order  derivatives of $H$ is involved in \eqref{lin-op1}, we will conclude that
     the constant $C_0$ in Theorem \ref{ex-una}  (ii)
     depends at most on $\lz$, $\Lambda$ ${\rm dist}(V,\partial U) $ and $\|u^{\gz}\|_{C^0(\partial U)}$.
          To get  Theorem \ref{ex-una} (iii), we need new ideas.
    Indeed, unlike the case  $H(p)=\frac12|p|^2$, where we use
     $|x|^\gz$ as a barrier function to conclude Theorem \ref{ex-una} (iii) from the comparison principle,
      the novelty here is that due to we take viscosity solutions $\mathcal{L}^{\dz}_{\sigma}$
      of certain Hamilton-Jacobi equation as barrier functions; see Lemmas 3.1-3.2.
   %   For the reason see Remark \ref{rem}

In Section 5, we establish the  following flatness estimate
of $u^{\gz,\ez}$, which is  is crucial to show that $\mathscr Du(x)$ is a singleton.
Denote by $e_n$    the vector $(0,\cdots,0,1)$.

  \begin{thm}\label{Fla}
Suppose that $U=B(0,3)$ and for some $\gz,\ez\in(0,1]$, $u^{\gz,\ez}$ satisfies
\begin{equation}\label{cond1}
\max_{B(0,3)}|u^{\gz,\ez}-  x_n|\le \tau
\end{equation}
for some $0<\tau<1$  and
\begin{equation}\label{H(Du)1}
 H^\gz(Du^{\gz,\ez} (x_0))\ge H^\gz(e_n)- \delta
\end{equation}
for some $0<\dz< H(e_n)/2 $ and $x_0\in B(0,1)$.
  Then
\begin{align}\label{flat-ex}
|Du^{\gz,\ez}(x_0)-e_n|^2&\le C_1(\lambda,\Lambda)\left[\tau+\dz+\frac1{\ez}e^{-\frac{\mu\delta}{\ez}} \right],
 \end{align}
 where $\mu=\frac\lz {16n}$.
 Above $C_1(\lambda,\Lambda)$ is a constant depending only on $\lz$ and $\Lambda$.
\end{thm}

The proof of Theorem  \ref{Fla}  relies on a generalization  of
 the adjoint method of Evans-Smart \cite{es11b} to the equation \eqref{apeq2} as developed in Section 4.
 Moreover, since $H$ does not have Hilbert structure necessarily, we can not follow the argument of Evans-Smart
 to get Theorem \ref{Fla}, where they take  $ u^{\gz,\ez}_{x_n}-1$ as a test function.
 The novelty here is  to
 take $|Du^{\gz,\ez}-e_n|^2$ as  a test function.  With aid of the
 estimates in Section 4,  by using the strongly convexity/concavity
 of $H$   and some careful analysis, we are able to prove Theorem \ref{Fla}.
 Again, since   none of $k\ge 3$ order  derivatives of $H$ are involved in the
linearized operator and hence in  the whole procedure,  we conclude that
all constants in estimates  in Section 4 and hence $C_1$ in Theorem \ref{Fla}
 depend on at most $\lambda$ and $\Lambda$.

With the aid of Theorem \ref{Fla}, Theorem \ref{ex-una} and Lemma \ref{lla},
by some necessary modifications of the arguments  of \cite{es11a}
we are able to prove that for any $x\in\Omega$,
 $\mathscr Du(x)$ is singleton, and hence that $u$ is differentiable everywhere in $\Omega$;
  for reader's convenience we give the details. % in Section 2.1%; but for reader familiar with \cite{es11a}, Section 2.1 can be ignored.

\begin{proof} [Proof of Theorem \ref{Theorem}]
By Remark \ref{r1.3}, we assume     $H\in C^0(\rn)$ satisfies (H1)\&(H2).
Let $\Omega$ be any domain of $\rn$, and    $u\in AM_H(\Omega) $.
It suffices to prove that $\mathscr Du(x)$ is   singleton.
%Assume that   $H$ satisfies (H1)\& (H2) as below.  Let $u\in AM_B(\Omega)$ for some domain $\Omega\subset\rn$.
We prove this by contradiction.
Assume that  $\mathscr Du(x_0)$ contains at least two vectors $a\ne b$ with $H(a)=H(b)$ for some $x_0\in\Omega$.
Note that $a,b\ne0$.  We may  assume that $x_0=0\in \Omega$, $u(0)=0$, $a=e_n$ without loss of generality.
Set $\theta=|b-e_n|>0$. We obtain a contradiction by the following 4 steps.

\medskip
\noindent{\it Step 1.}
Fix    $\tau _\theta\in(0,1]$  such that
\begin{equation}\label{in1xx}
C _1(\lz,\Lambda)  \tau^{1/2} \le  \frac{\theta^2}{32},\quad \quad\forall   \dz< \dz_{  \tz}, \tau<\tau_\theta.
\end{equation}
Since $e_n\in\mathscr Du(0)$ we can find a sequence $\{r_j\}_{j\in\nn}$ which converges to $0$ such that
\begin{equation}\label{m-1}
\max_{B(0,3 )} {|u_j(x)- x_n|} =\max_{B(0,3r_j)}\frac{|u(x)-a\cdot x|}{r_j}\rightarrow 0 \quad\mbox{as $j\to\fz$},
\end{equation}
where  $u_j(x)=u(r_jx)/{r_j}$ for   $x\in \frac1{r_j}\Omega$.
For each  $\tau \in(0,\tau _\theta]$, there exists a $j_\tau $ such that if $j\ge j_\tau $,
 \begin{equation}\label{m-j}
\max_{B(0,3 )} {|u_j(x)- x_n|} <\frac\tau 4\quad\forall j\ge j_\tau .
\end{equation}

For any $\gz\in(0,1)$ and $j\in\nn$, let
 $$\mbox{$u^\gz_j\in AM_{H^\gz}(B(0,3))$ with $u^\gz_j=u_j$ on $\partial B(0,3)$.}$$
By Lemma \ref{abso},
for each $j\ge j_\tau$, $u_j^\gz\to u_j$ as $\gz\to0$,
 there exists $\gamma_{j,\tau }>0$ such that
$$\max_{B(0,3)}|u^{\gamma}_j-x_n|\le \frac{\tau }{2},\quad
\forall \gamma\in(0,\gamma_{j, \tau }].$$
By Theorem \ref{ex-una} (iv), for each $j\ge j_\tau$ and $\gz<\gamma_{j,\tau }$, there
 is an $\ez_{ \gz,j,\tau   }\in(0,1]$  that
\begin{equation}\label{m-0}\max_{B(0,3)}|u^{\gamma,\ez}_j-x_n|\le \tau \quad\forall \ez<\ez_{\gz,j,\tau }.
\end{equation}

\medskip
\noindent{\it Step 2.}
 Since $b\in\mathscr D(u)(0) $
by \cite{fwz}, there exist a
sequence    $\{s_k\}_{k=1}^\infty$    which converge to zero
such that
$$
\max_{B(0,s_k/r_j)}\frac{|u_j(x)-b\cdot x|}{s_k/r_j}=
\max_{B(0,s_k)}\frac{|u(x)-b\cdot x|}{s_k}\rightarrow 0\quad \mbox{ as $k\to\fz$}
$$
for all $j\in\nn$.
For  each $\eta\in(0,\tau)$ and $j\ge j_\tau$,
 there exist $ k_{\eta,j}\in\nn $  such that for all $ k\ge   k_{\eta,j}$, we have $s_k/r_j\le1$
 and
$$\max_{B(0,s_k/r_j)}\frac{|u_j-b\cdot x|}{s_k/r_j}\le \frac{\eta}{4}.$$
Since $u^\gz_j\to u_j$ in $B(0,3)$, for each $ k\ge   k_{\eta,j}$ we can find $\gz_{k,j,\eta}<\gz_{j,\tau}$ such that
for each $\gz\in(0, \gz_{k,j,\eta})$, $$\max_{B(0,s_k/r_j)}\frac{|u_j^\gz-b\cdot x|}{s_k/r_j}\le \frac{\eta}{2}.$$
Since $u^{\gz,\ez}_j\to u^\gz_j$ in $B(0,3)$, for each $\gz\in(0, \gz_{k,j,\eta})$, we further find $\ez_{\gz,k,j,\eta}<\ez_{\gz,j,\tau}$ such that for all $\ez<\ez_{\gz,k,j,\eta}$,
\begin{equation}\label{m-2}\max_{B(0,s_k/r_j)}\frac{|u_j^{\gz,\ez}-b\cdot x|}{s_k/r_j}\le  \eta .
\end{equation}

\medskip
\noindent{\it Step 3.} For  each $\eta\in(0,\tau)$, there exists $\gz_\eta\in(0,1)$ such that
$$
 |H^\gz(e_n)-H^\gz(b)|\le   \eta,$$
where we have used $H(e_n)=H(b)$.\\
For  each $\eta\in(0,\tau)$,  $j\ge j_\tau$, $k\ge k_{j,\eta}$, $\gz<\min\{\gz_{k,j,\eta},\gz_\eta\}$
and $\ez<\ez_{\gz,k,j,\eta}$, by Lemma \cite{es11a,es11b}, \eqref{m-2} implies that there is a point $ x_{\ez,\gz,k,j,\eta}\in B(0,s_k/r_j)$ at which
\begin{equation}\label{key1}
|Du^{\gamma,\ez}_j(x_{\ez,\gz,k,j,\eta})-b|\le 4\eta.
\end{equation}

We further have  that
\begin{equation}\label{key2}
|H^\gz(e_n)-H^\gz(Du^{ \gamma,\ez}_j(x_{\ez,\gz,k,j,\eta}))| \le  C_2(\lambda,\Lambda,b)\eta.\end{equation}
Indeed, by convexity of $H$, we have
$$H^\gz(b)-H^\gz(Du^{\gamma,\ez }_j(x_{\ez,\gz,k,j,\eta}))\ge  \langle D_pH(Du^{\ez,\gamma}_j(x_{\ez,\gz,k,j,\eta})), b- Du^{ \gamma,\ez}_j(x_{\ez,\gz,k,j,\eta})\rangle.$$
Since $|D_pH^\gz(p)|\le \Lambda |p|$ and $|Du^{\gamma,\ez}_j(x^0)|\le |b|+4$, one has
$$  H^\gz(Du^{\ez,\gamma}_j(x_{\ez,\gz,k,j,\eta}))-H^\gz(b)\le
 |D_pH^\gz(Du^{\ez,\gamma}_j(x_{\ez,\gz,k,j,\eta}))|  |Du^{\gamma,\ez }_j(x_{\ez,\gz,k,j,\eta})-b|\le C (\lambda, \Lambda,b)  \eta.$$
A similar estimate holds for $H^\gz(b)-H^\gz(Du^{\ez,\gamma}_j(x_{\ez,\gz,k,j,\eta}))$.  Thus
$$  |H^\gz(Du^{\ez,\gamma}_j(x_{\ez,\gz,k,j,\eta}))-H^\gz(b)|  \le C (\lambda, \Lambda,b)  \eta.$$
This together with  $|H^\gz(e_n)-H^\gz(b)|\le\eta$ implies \eqref{key2}.

\medskip
\noindent{\it Step 4.} Let $\dz_{ \theta}\in(0,1]$    such that
\begin{equation}\label{in1x}
C _1(\lz,\Lambda) \delta \le  \frac{\theta^2}{32},\quad \quad\forall   \dz< \dz_{  \tz};
\end{equation}
For each $ \mu\in(0,\frac{\lambda }{8n}]$, $\dz\in(0,\dz_\theta]$, let $\ez_{\mu,\dz,\theta} \in(0,1]$ such that
\begin{equation}\label{in3}
C _1(\lz,\Lambda)\frac{1}{\ez}e^{-\frac{\mu\delta}{\ez}}
\le \frac{\theta^2}{32},\quad \forall \ez\in (0,\ez_{\mu,\dz,\theta}].
\end{equation}
Let $  \eta<\min\{\dz_{ \theta}/C_2(\lambda,\Lambda,b), \tz/{16}\} $ and $\dz=C_2(\lambda,\Lambda,b)\eta$. For $\tau<\tau_\tz$,
  $j\ge j_\tau$, $k\ge k_{j,\eta}$, $\gz<\min\{\gz_{k,j,\eta},\gz_\eta\}$,
  and $\ez<\min\{\ez_{\gz,k,j,\eta},\ez_{\mu,\dz,\theta}\}$, by Theorem \ref{Fla}, \eqref{key2} and \eqref{m-0} imply that
\begin{align}\label{flat-exx}
|Du^{\gz,\ez}(x_{\ez,\gz,k,j,\eta})-e_n|^2&\le C_1(\lambda,\Lambda)\left[\tau+C_2(\lambda,\Lambda,b)\eta+\frac1{\ez}e^{-\frac{\mu\delta}{\ez}} \right]\le \frac{\theta^2}{8}.
 \end{align}
Thus by \eqref{key1} one has
$$\tz=
|e_n-b|\le 4\eta+\frac\tz 2\le \frac{3\tz}4, $$
which is a contradiction as desired. The proof of Theorem \ref{Theorem} is complete.
\end{proof}

To prove Theorem \ref{Theorem2}, let $u\in AM_H(\rn)$ with a linear growth at  $\infty$.
By \cite{fwz}, $\|Du\|_{L^\fz(\rn)}<\fz$ and moreover  $u$ has the linear approximation property at $\fz$ as below.
 \begin{lem} \label{llafz}
For   any   sequence $\{r_j\}_{j\in\nn}$ which converges to $\fz$,
there exist a subsequence $
  \{r_{j_k}\}_{k\in\nn}$ and a vector $e_{\{r_{j_k}\}_{k\in\nn}}$ such that    $$\lim_{k\to\fz}\sup_{y\in B\left(0,1\right)}\left|\frac{u ( r_{j_k}y ) }{r_{j_k}} -  e_{\{r_{j_k}\}_{k\in\nn}}\cdot y \right|=0 $$
  and $$H (e_{\{r_{j_k}\}_{k\in\nn}} ) =  \|H\left(Du\right)\|_{L^\fz\left(\rn\right)}.$$
  \end{lem}
Denote by $\mathscr D(\fz)$ the collection of all possible $e_{\{r_{j_k}\}}$ as above.
%To get Theorem \ref{Theorem2} it suffices to show that $\mathscr D(\fz)$ is singleton.
Following the   proof of    Theorem \ref{Theorem} line by line and letting $r_k\to \fz$ as $k\to\fz$,  we are able to prove
that $\mathscr D(\fz)$ is singleton, and hence prove Theorem \ref{Theorem2}; here we omit the details
and also refer to \cite{hz,mwz2}.

We  end this section by the following remark.

\begin{rem}\label{rem}\rm
(i) Recall that Evans \cite{e03} suggested  another approximation to $u^\gz$
via $e^{\frac1\ez H}$-harmonic functions $\hat u^{\gz,\ez}$, that is, smooth solutions  to
$$
   {\,\rm div\,}( e^{\frac1\ez H(Dv)} D_pH(Dv))=
   e^{\frac1\ez H(Dv)}[\mathscr{A}_{H^{\gamma}} [v]+\ez{\,\rm div\,}(  D_pH^\gz(Dv)) ]=0\quad {\rm in \ }\
U;\ v=u^\gz\ {\rm on} \ \partial U.
$$
%Note that when $ n=2$, such approximation was used to prove the Sobolev regularity in \cite{fmz}.
But  note that  the 3-order derivative of $H^\gz$ appears in third terms of the linearized operator
  \begin{equation}\label{lin-op2}
-
H_{p_i}(D\hat u^{\gz,\ez})H_{p_j}(D\hat u^{\gz,\ez})v_{x_ix_j}-2H_{p_ip_l}(D\hat u^{\gz,\ez})\hat u^{\gz,\ez}_{x_ix_j}H_{p_j}(D\hat u^{\gz,\ez})
v_{x_l}- \ez{\,\rm div\,}(  D_{pp}^2H^\gz(D\hat u^{\gz,\ez})Dv).
\end{equation}
 If we want to get  Theorem \ref{ex-una} and \ref{Fla}  for $\hat u^{\gz,\ez}$ so that the constants $C_0, C_1$ are independent of
 $3$-order derivative of $H^\gz$ or $H$, some extra efforts are needed.
% This will bering some other difficulty to
% get rid of the dependence on the $3$-order derivative of $H^\gz$ or $H$
 %
%Thus the constants appearing in Theorem \ref{ex-una} (ii),  estimates  , and in Theorem \ref{ex-una} will rely on the
%3-order derivative of $H^\gz$.
%If $H\in C^{2,1}(\rn)$ is locally strongly convex, one would obtain Theorem \ref{Theorem}.
%But when $H\in C^{2,1}(\rn)$ is merely locally strongly convex/concave, we can not obtain  Theorem \ref{Theorem}
%via such approximation.
To avoid such extra efforts,  % to get rid of the dependence on the 3-order derivatives,
we prefer to
 consider the approximation equation \eqref{apeq2}.

 (ii)
 If $H(p)=\frac12 |p|^2$,
 a flatness estimate stronger than Theorem \ref{Fla} is also given in \cite{es11a}
  via the maximal principle,
 \begin{equation}\label{E2.xx}|Du^{\gz,\ez}|^2\le  u^{\gz,\ez}_{x_n} +C\sqrt\tau\quad\mbox{in $B(0,1 )$ for all $\ez\in(0,1]$}.
 \end{equation}
 Note that in this case, $H^\gz=H$  and $u^\gz=u$, and $u^{\gz,\ez}$ is then reduced to $u^\ez$.
From this Evans-Smart \cite{es11a} concluded the everywhere differentiability of $\fz$-harmonic functions $u$.
  But for $H\in C^0(\rn)$ satisfying (H1) and (H2), since $H$ does not necessarily have a Hilbert structure,
 it is still unclear
 whether there is some  estimate similar to \eqref{E2.xx}, and also whether the approach in \cite{es11a} can be used to prove Theorem \ref{Theorem}.
 \end{rem}

\section{Proof of Theorem \ref{ex-una}}

Let $H$, $H^\gz$, $u$  $u^\gz$   and $u^{\gz,\ez}$  be as in Section 2.
Note that $H^\gz$ satisfies (H1)\&(H2) with the same $\lz$ and $\Lambda$.
Since $ |H^\gz_p(p)|^2\le \Lambda^2|p|^2$ implies
$$\ez|\xi|^2\le [H^\gz_{p_i}(p)H^\gz_{p_j}(p)+\ez \dz_{ij}]\xi_i\xi_j \le \Lambda^2(|p|^2+1)|\xi|^2\quad\forall \xi\in\rn,$$
 by a standard quasilinear elliptic theory (see \cite{gt}), there exists a unique smooth solution
$u^{\gz,\ez}\in C^\fz(U)\cap C^0(\overline U)$
to   \eqref{apeq2}.   Theorem \ref{ex-una} (i)  follows from the known maximum principle.
 We also note that by a standard argument,
 $u^{\gz,\ez}\to u^\gz$ in $C^0(\overline U)$ (that is Theorem \ref{ex-una} (iv))   follows from
  Theorem \ref{ex-una} (ii)\&(iii), and the uniqueness of $u^\gz$ in \cite{acjs}; here we omit the details.
  Below we only need to prove Theorem \ref{ex-una} (ii)\&(iii).
 For simplicity, we write $H^\gz$ as $H$,   $u^\gz$ as $u$,
 and  we write $u^{\gz,\ez}$ as $u^\ez$ by abuse of notation.

%%Let $U=B(0,3)$ and
%%   $u\in AM_H(U)\cap  C^{0,1}(\overline U)$.
%   For each $\ez\in(0,1]$, the equation \eqref{apeq2} reads as %For every $\ez\in(0,1]$, we consider the regularized equations
%\begin{equation}\label{req}
%\mathscr{A}_H (v)+\ez \Delta v=[H_{p_i}(p)H_{p_j}(p)+\ez \dz_{ij}]v_{ij}=0\quad {\rm in \ }\
%U;\ v=u\ {\rm on} \ \partial U.
%\end{equation}

We  prove Theorem \ref{ex-una} (ii) using the approach of Evans-Smart \cite{es11b} here.
Denote by $L_\ez$  the linearized operator obtained from $\mathscr A_H[u^\ez]+\ez \Delta u^\ez=0$, that is,
\begin{equation}\label{lin-op}
L_\ez (v):=-
H_{p_i}(Du^\ez)H_{p_j}(Du^\ez)v_{x_ix_j}-2H_{p_ip_l}(Du^\ez)u^\ez_{x_ix_j}H_{p_j}(Du^\ez)
v_{x_l}-\ez \Delta v
\end{equation}
for $ v\in C^\fz(U)$.
Note that
$$\mbox{$L_\ez (u^\ez_{x_k})=(\mathscr A_H[u^\ez]+\ez \Delta u^\ez)_{x_k}=0$ in $U$ for all $k=1,\cdots,n$.}$$

\begin{proof}
[Proof of  Theorem \ref{ex-una} (ii)]
We choose $\zeta\in C^\fz_c(U)$ such that
$$0\le \zeta\le 1,\quad \zeta=1\ {\rm in }\ \overline V, \quad |D\xi|\le 4\frac 1{\dist(V,\partial U)},
\quad |D^2\xi|\le C_0\frac 1{[\dist(V,\partial U)]^2}.$$
  Define an auxiliary function
$$w=\zeta^2  |Du^\ez|^2 +\alpha  (u^\ez)^2 ,$$
where $\alpha>0$ will be determined later.
If $w$ attains its maximum on $\partial U$, then
$$\max_{\overline V}|Du^\ez|^2\le \sup_{U}w=\max_{\partial U} {\alpha u^2} ,$$
this implies Theorem \ref{ex-una} (ii).

Assume that $w$ attains its maximum
at some $x_0\in U$. Since $Dw(x_0)=0$ and
 $D^2 w(x_0)  $  is nonpositive definite, we have $L_\ez (w) \ge 0$ at $x_0$.
Below we   estimate $L_\ez (w)$ at $x_0$ from above.
Note that
\begin{align*}
 L_\ez(w)&=\zeta^2L_\ez (|Du^\ez|^2)+|Du^\ez|^2 L_\ez(\zeta^2)+\alpha
L_\ez ((u^\ez)^2)\\
&\quad +[-4\zeta\langle D_pH(Du^\ez),D\zeta\rangle\langle D^2u^\ez D_pH(Du^\ez),Du^\ez\rangle -8\ez \zeta \langle D^2u^\ez Du^\ez,D\zeta\rangle].
\end{align*}
A direct calculation gives
\begin{align*}
L_\ez \left( {|Du^\ez|^2} \right)=&-2H_{p_i}(Du^\ez)H_{p_j}(Du^\ez)
[u^\ez_{x_k}u^\ez_{x_kx_ix_j}+u^\ez_{x_kx_i}u^\ez_{x_kx_j}]\nonumber\\
&-4H_{p_ip_l}(Du^\ez)H_{p_j}(Du^\ez)u^\ez_{x_ix_j}u^\ez_{x_k}u^\ez_{x_kx_l}-2
\ez[u^\ez_{x_k}u^\ez_{x_kx_ix_i}+(u^\ez_{x_kx_i})^2]\nonumber\\
=&2u^\ez_{x_k}L_\ez(u^\ez_{x_k})-2|D^2u^\ez D_pH(Du^\ez)|^2-2\ez|D^2u^\ez|^2.
\end{align*}
By $L_\ez (u^\ez_{x_k})=0$, and $D^2u^\ez D_pH(Du^\ez)=D[H(Du^\ez)]$ we obtain
\begin{align*}%\label{es1a}
\zeta^2L_\ez \left( |Du^\ez|^2 \right)
&=-2\zeta^2| D [H(Du^\ez)]|^2-2\ez\zeta^2|D^2u^\ez|^2.\nonumber
\end{align*}
Similarly using \eqref{apeq2}, we have \begin{align*}
L_\ez \left( (u^\ez)^2 \right)=&-2H_{p_i}(Du^\ez)H_{p_j}(Du^\ez)
[u^\ez u^\ez_{x_ix_j}+u^\ez_{x_i}u^\ez_{x_j}]\nonumber\\
&-4 u^\ez H_{p_ip_l}(Du^\ez)H_{p_j}(Du^\ez)u^\ez_{x_ix_j}u^\ez_{x_l}-2
\ez[u^\ez u^\ez_{x_ix_i}+(u^\ez_{x_i})^2]\nonumber\\
 =&-2\langle D_pH(Du^\ez),Du^\ez\rangle^2-2\ez |Du^\ez|^2
-4u^\ez \langle D^2_{pp} H(Du^\ez)  D [H(Du^\ez)],Du^\ez\rangle\nonumber.
\end{align*}
Since  (H1)\&(H2) implies
$$\langle D_pH(p),p\rangle\ge \frac{\lambda}{2}
|p|^2,\ |D^2_{pp} H(p)\xi|\le \Lambda|\xi|\quad \forall p,\xi\in\rn,$$
 by Young's inequality,  we   obtain
\begin{align*}
\alpha L_\ez \left( (u^\ez)^2 \right)
&\le -\alpha \lambda ^2 |Du^\ez|^4+C(\alpha,\Lambda)  | D [H(Du^\ez)]|^{4/3}+
|u^\ez|^4 | Du^\ez  |^4.
\end{align*}
Since
\begin{align*}%\label{es3a}
L_\ez \left(\zeta^2\right)=&-2H_{p_i}(Du^\ez)H_{p_j}(Du^\ez)
[\zeta \zeta_{x_ix_j}+\zeta_{x_i}\zeta_{x_j}]\nonumber\\
&-4 \zeta H_{p_ip_l}(Du^\ez)H_{p_j}(Du^\ez)u^\ez_{x_ix_j}\zeta_{x_l}-
2\ez[\zeta \zeta_{x_ix_i}+(\zeta_{x_i})^2],
\end{align*}
using (H1)\&(H2) and Young's inequality we also obtain
\begin{align*}%\label{es3a}
|Du^\ez|^2L_\ez \left(\zeta^2\right)
&\le C(\Lambda)|Du^\ez|^4 [|D\zeta|^2+|D^2\zeta|\zeta]+ \frac14 | D [H(Du^\ez)]|^2\zeta^2 +
C(\Lambda )[|D\zeta|^2+|D^2\zeta|\zeta].
\end{align*}

Similarly,
\begin{align*}
&-4\zeta\langle D_pH(Du^\ez),D\zeta\rangle\langle D^2u^\ez D_pH(Du^\ez),Du^\ez\rangle -8\ez \zeta \langle D^2u^\ez Du^\ez,D\zeta\rangle\\&\quad\le \frac14 \zeta^2 |D[H(Du^\ez)]|^2 +\frac14 \ez \zeta^2|D^2u^\ez|^2
+C(\Lambda)|Du^\ez|^4+C(\Lambda)|\zeta|^2.
\end{align*}
In conclusion, we have
\begin{align*}  L_\ez(w)\le&  -\zeta^2 |D[H(Du^\ez)]|^2 + C(\alpha,\Lambda) |D[H(Du^\ez)]|^{4/3}
  -[\alpha\lambda^2-
 C(\Lambda)(|D\zeta|^2+|D^2\zeta|\zeta)]|Du^\ez|^4\\
 &- |u^\ez|^4|Du^\ez|^4 +C(\Lambda )[|D\zeta|^2+|D^2\zeta|\zeta]+C(\Lambda)|\zeta|^2.\end{align*}

At $x_0$,  $L_\ez(w)\ge 0$ implies that
\begin{align*}      &\zeta^2 |D[H(Du^\ez)]|^2 + [\alpha\lambda^2-
 C(\Lambda)(|D\zeta|^2+|D^2\zeta|\zeta)]|Du^\ez|^4\\
 &\quad\le   C(\alpha,\Lambda) |D[H(Du^\ez)]|^{4/3}+ |u^\ez|^4|Du^\ez|^4
  +C(\Lambda )[|D\zeta|^2+|D^2\zeta|\zeta].\end{align*}
Multiplying the above inequality with $\zeta^4$ yields
\begin{align*}      &|D[H(Du^\ez)]|^2\zeta^6 + [\alpha\lambda^2-
 C(\Lambda)(|D\zeta|^2+|D^2\zeta|\zeta)]|Du^\ez|^4\zeta^4\\
 &\quad\le   C(\alpha,\Lambda) |D[H(Du^\ez)]|^{4/3}\zeta^4 + |u^\ez|^4|Du^\ez|^4\zeta^4
  +C(\Lambda )[|D\zeta|^2+|D^2\zeta|\zeta]\zeta^4.\end{align*}
By Young's inequality we have
$$C(\alpha,\Lambda) |D[H(Du^\ez)]|^{4/3}\zeta^4 \le \frac12 |D[H(Du^\ez)]|^{2}\zeta^6 + C(\alpha,\Lambda) ,$$
and hence
\begin{align*}         [\alpha\lambda^2-
 C(\Lambda)(|D\zeta|^2+|D^2\zeta|\zeta)]|Du^\ez|^4\zeta^4
 &\le   |u^\ez|^4|Du^\ez|^4\zeta^4
  +C(\Lambda )[|D\zeta|^2+|D^2\zeta|\zeta ]\zeta^4+C(\alpha,\Lambda) .\end{align*}
Choosing $\alpha=\alpha ( \lambda,\Lambda,\|u^\ez\|_{C^0(U)},\dist(V,\partial U) )$   so that
 $$
 \alpha\lambda^2-
  C(\Lambda) \frac{C_0+16}{[\dist(V,\partial U)]^2}\ge \|u^\ez\|_{C^0(U)}^4+1,$$
 we have
$$\zeta^4|Du^\ez|^4|_{x=x_0}\le C(\lambda,\Lambda,\|u^\ez\|_{C^0(U)},\dist(V,\partial U)).$$
Hence,
$$\sup_{V}|Du^\ez|^4\le [\sup_{U}  w]^2\le \zeta^4|Du^\ez|^4|_{x=x_0}+ \alpha[ u^\ez(0)]^2\le C(\lambda,\Lambda,\|u^\ez\|_{C^0(U)},\dist(V,\partial U))    $$
as desired.
\end{proof}
%Owing to Theorem \ref{ex-un}, it is easy to check that there exists $\wz u\in C^0(B(0,3))$ such that
%$$u^\ez\to \wz u\ {\rm in}\
%C^0_{\loc}(B(0,3))\ {\rm as}\ \ez\to 0.$$
%If we can proof the $u^\ez \to \wz u$ in $C^0(\overline B(0,3))$ as
%$\ez \to 0$, by uniqueness of viscosity  solution we have
%$\wz u=u$ in $B(0,3)$ with same boundary value. Thus it suffice to
%study the behavior of $u^\ez$ near $\partial B(0,3)$. To do this,
%we borrow the semi-concave function $\mathcal{ L}^\lambda_\sz(x_0,x)$ and comparison principle to
%obtain the Theorem \ref{holder}.

To prove Theorem \ref{ex-una} (iii), we need the following Lemma \ref{distance0}, which can be found in \cite[Lemma 3.2 and Lemma 3.4]{fm}.   For each $t>0$, $\dz>0$, $\sz>0$ and $x,y\in U$,
define
$$\mathcal{L}^{\dz}_{\sigma}(x,y):=\inf\Big\{\int^{t}_{0}
\left[\sz+L(\dot{\xi}(s))\right] e^{-\dz(t-s)}\, ds\Big| t>0,\xi\in\mathcal{C}(0,t;x,y;U)\Big\},$$
where   $\mathcal{C}(0,t;x,y; U)$  is the set of all rectifiable curves $\xi:[0,t]\rightarrow  U$ that joins $x $ to $y$,  and
$$L(q) =\sup_{p\in \mathbb{R}^n}\{p\cdot q-H(p)\},\quad \forall q\in\mathbb{R}^n.$$
For each $\sz>0$, we also need to the notion of generalized cones, that is
$$\mathscr C^H_{\sz}(x)=\max_{H(p)=\sz}\{p\cdot x\},\quad \forall x\in U.$$
By the strongly convexity of $H$, one always has that  $$\sqrt{2\sz/\Lambda}|x|\le \mathscr C^H_{\sz}(x)\le  \sqrt{2\sz/\lz}|x|$$
%Moreover,   $\wz H(p)=H(-p)$ for any $p\in \rr^n$. We write
%$\mathcal{\wz L}^\lambda_\sz$ and $C^{\wz H}_\dz$ in above notions.

\begin{lem}\label{distance0}
Assume that $H\in C^\fz(\rr^n)$ satisfy (H1){\rm\&}(H2).
 \begin{enumerate}
\item[(i)]
For all $\sz>0$,$\dz\geq0$ and $x,y\in U$,  we have
\begin{equation*}
  \mathscr C^H_\sigma(y-x)\ge \mathcal{L}^\dz_{\sigma}(x,y)\geq0.
\end{equation*}

 \item[(ii)] When $\frac{\dz}{\sigma}\mathcal{L}^{\dz}_{\sigma}(x,y)<\ln \sqrt {2}$, we also have
\begin{equation*}
\mathscr C^H_\sigma (y-x)\leq e^{\frac{4\dz}{\sigma}\mathcal{L}^{\dz}_{\sigma}(x,y)}\mathcal{L}^{\dz}_{\sigma}(x,y).
\end{equation*}

 \item[(iii)] For any
domain $V\Subset \rr^n$ and $x _0\in \overline V$, we have
 $\mathcal{L}^\dz_\sz(x_0,\cdot)$ is a viscosity sup-solution of
$$\mbox{$\mathscr{A}_H (v)=-\frac{\dz \sz}{2}$ in $V\backslash \{x_0\}$  whenever } 0<\dz<\dz_{\sz,V}=\frac{\sz}{2\sup\{\mathscr C^H_{\sz}(y-x):x,y\in \overline V\}}$$

 and
  $\mathcal{L}^\dz_\sz(\cdot, x_0)$ is a viscosity sub-solution of
$$\mbox{$\mathscr{A}_H (v)= \frac{\dz \sz}{2}$ in $V\backslash \{x_0\}$  whenever } 0<\dz<\dz_{\sz,V}.$$
 \end{enumerate}
\end{lem}
%Moreover, with aid of Lemma \ref{distance0} and Lemma \ref{distance1}, we need to
%prove two comparison principle of viscosity solution.

We also need the following comparison principle, see \cite[ Appendix,
Theorem 2]{y06}.
\begin{lem}\label{com-pr}
Assume that $H\in C^\fz(\rr^n)$ satisfy (H1){\rm \&}(H2).
 For any $\sz>0$ and domain $V\subset \rr^n$, assume that $u_1\in C^{0}(\overline V)$ is a viscosity sup-solution of
$$\mathscr{A}_H (u_1)+\ez \Delta u_1=-\dz\quad in \ V$$
and $u_2\in C^{0}(\overline V)$ is a viscosity sub-solution of
$$\mathscr{A}_H (u_2)+\ez \Delta u_2=0\quad in \ V.$$
If either $u_1\in C^{0,1}(V)$ or $u_2\in C^{0,1}(V)$, then
$$\max_{\overline V}(u_2-u_1)\le \max_{\partial V}(u_2-u_1).$$
\end{lem}

From Lemma \ref{distance0} and \ref{com-pr},  we deduce the following.
\begin{lem}\label{com-pr1}
Assume that $H\in C^\fz(\rr^n)$ satisfy (H1){\rm\&}(H2). For any domain $V\Subset\rn$ and
 $x_0\in \overline V$
  for all   $\sz>0$  and $0<\dz<\dz_{\sz,V}$,
  there exist constant $\mu_1,\mu_2>0 $ depending on $\sz, \dz, H$   such that for all $\ez\in(0,1)$,
$\mathcal{L}^\lambda_\sz(x_0, \cdot)$ is a viscosity sup-solution of
$$\mbox{$\mathscr{A}_H (v)+\ez \Delta v=-\frac{\dz \sz}{2}+\ez n\mu_1$ in $V\backslash \{x_0\}$}$$
and $ \mathcal{  L}^\lambda_\sz(\cdot,x_0)$ is a viscosity sub-solution of
$$\mbox{$\mathscr{A}_{H} (v)+\ez \Delta v=\frac{\dz \sz}{2}-\ez n\mu_2$ in $V\backslash \{x_0\}$.}$$
\end{lem}

\begin{proof}
[Proof of Lemma \ref{com-pr1}]
For any $\phi\in C^2(V)$ and $\mathcal{L}^\dz _\sz(x_0,x)-\phi(x)$ attains its locally minimum
at $y\in V\backslash \{x_0\}$, it suffice to prove that
\begin{equation}\label{com0}
\mathscr{A}_H(\phi)(y)+\ez \Delta \phi(y)\le -\frac{\dz  \sz}{2}+\ez n\mu_1.
\end{equation}
Without loss of generality, we may assume that $\mathcal{L}^\dz _\sz(x_0,x)-\phi$ attains
its  a strictly minimum at $y_0\in V\backslash \{x_0\}$. Since $\mathcal{L}^\dz _\sz(x_0,x)$
is semiconcave, for any $\eta>0$, $r>0$, by Lemma A.3 in \cite{CIL} there exist $x^{r,\eta}\in B(y_0,r)$
and $p^{r,\eta}\in B(0,\eta)$ such that $\mathcal{L}^\dz _\sz(x_0,x)-\phi(x)-
\langle p^{r,\eta},x\rangle$ has a local minimal at $x^{r,\eta}$ and $\mathcal{L}^\dz _\sz(x_0,x)$ is twice differentiable at $x^{r,\eta}$. Also, the semiconcave property of
$\mathcal{L}^\dz _\sz(x_0,x)$ implies that there exists $\mu_1>0$ depending on $\sz,\dz,H$ such that
\begin{equation}\label{com3}
D^2 \mathcal{L}^\dz _\sz(x_0,x)\le \mu_1 I_n
\end{equation}
in the sense of distributions, where $I_n$ is identity matrix. Since $\mathcal{L}^\dz _\sz(x_0,x)$
is twice differentiable at $x^{r,\eta}$, by Lemma \ref {distance0} and \eqref{com3}, we have
\begin{equation}\label{com4}
\langle D^2 \mathcal{L}^\dz _\sz(x_0,x^{r,\eta}) D_pH(D\mathcal{L}^\dz _\sz(x_0,x^{r,\eta})),D_pH(D\mathcal{L}^\dz _\sz(x_0,x^{r,\eta}))\rangle
+\ez \Delta \mathcal{L}^\dz _\sz(x_0,x^{r,\eta})\le
-\frac{\dz \sz }{2}+\ez n \mu_1.
\end{equation}
On the other hand, since $\mathcal{L}^\dz _\sz(x_0,x)-\phi(x)-
\langle p^{r,\eta},x\rangle$ has a local minimal at $x^{r,\eta}$, we have $D\mathcal{L}^\dz _\sz(x_0,x^{r,\eta})=D
\phi(x^{r,\eta})+p^{r,\eta}$ and $D^2\mathcal{L}^\dz _\sz(x_0,x^{r,\eta})
\ge D^2\phi(x^{r,\eta})$. Thus
\begin{align}\label{com5}
&\langle D^2 \mathcal{L}^\dz _\sz(x_0,x^{r,\eta}) D_pH(D\mathcal{L}^\dz _\sz(x_0,x^{r,\eta})),D_pH(D\mathcal{L}^\dz _\sz(x_0,x^{r,\eta}))\rangle
+\ez \Delta \mathcal{L}^\dz _\sz(x_0,x^{r,\eta})\\
&\ge \langle D^2 \phi(x^{r,\eta}) D_pH(D\phi(x^{r,\eta})+p^{r,\eta}),D_pH(D\phi(x^{r,\eta})+p^{r,\eta})\rangle
+\ez \Delta \phi(x^{r,\eta}).\nonumber
\end{align}
Combing \eqref{com4} and \eqref{com5}, we have
\begin{align}\label{com6}
\langle D^2 \phi(x^{r,\eta}) D_pH(D\phi(x^{r,\eta})+p^{r,\eta}),D_pH(D\phi(x^{r,\eta})+p^{r,\eta})\rangle
+\ez \Delta \phi(x^{r,\eta})\le -\frac{\dz \sz }{2}+\ez n \mu_1.
\end{align}
Letting $r=\eta\to 0$ and noting $p^{r,\eta}\to 0$, $x^{r,\eta}\to y_0$,
this leads to the \eqref{com0}.
%Similarly, for some constant $\mu_2>0$ we have
%$$ D^2 \mathcal{\wz L}^\dz _\sz(x_0,x)\le \mu_2 I_n$$
%in the sense of distributions.
%That $\mathcal{\wz L}^\dz _\sz(x_0,x)$ is a viscosity sup-solution of $\mathscr{A}_{\wz H} (v)+\ez \Delta v=\frac{\dz  \sz}{2}-\ez n\mu_2$ in $V\backslash \{x_0\}$ follows analogously.
%Note that $D_p\wz H(p)=-D_pH(-p)$, we have
%\begin{align*}
%\mathscr{A}_{\wz H} (\mathcal{\wz L}^\dz _\sz(x_0,x))+\ez \Delta \mathcal{\wz L}^\dz _\sz(x_0,x)&=
%\wz H_{p_i}(D\mathcal{\wz L}^\dz _\sz(x_0,x))\wz H_{p_j}(D\mathcal{\wz L}^\dz _\sz(x_0,x))
%(\mathcal{\wz L}^\dz _\sz(x_0,x))_{x_ix_j}+\ez \Delta \mathcal{\wz L}^\dz _\sz(x_0,x)\\
% &=H_{p_i}(-D\mathcal{\wz L}^\dz _\sz(x_0,x)) H_{p_j}(-D\mathcal{\wz L}^\dz _\sz(x_0,x))
%(\mathcal{\wz L}^\dz _\sz(x_0,x))_{x_ix_j}+\ez \Delta \mathcal{\wz L}^\dz _\sz(x_0,x)\\
%&=-\mathscr{A}_{H} (-\mathcal{\wz L}^\dz _\sz(x_0,x))
%+\ez  \Delta \mathcal{\wz L}^\dz _\sz(x_0,x)\le \frac{\dz  \sz}{2}-\ez n\mu_2
%\end{align*}
%in viscosity sense in $V\backslash\{x_0\}$. This implies that

Similarly, we can prove  that
$-\mathcal{  L}^\dz _\sz(x, x_0)$ is viscosity sub-solution of
$$\mbox{$\mathscr{A}_{H} (v)+\ez \Delta v=\frac{\dz  \sz}{2}-\ez n\mu_2$ in $V\backslash \{x_0\}$.}$$
The proof is complete.
\end{proof}

We are able to prove Theorem 2.2 (iii) as below.

\begin{proof}
[Proof of   Theorem 2.2 (iii).]
Note that  $u\in C^{0,1}(\overline U)$. Letting $\sz>8\Lambda \|u\|^2_{C^{0,1}(\overline U)} $, we have
\begin{equation}\label{claim0}
|u(y)-u(x )|\le \|u\| _{C^{0,1}(\overline U)}|y-x| \le \frac14\mathscr C_\sz^H(y-x),\quad
\forall  x,y\in  U.
\end{equation}
Moreover, there exist $\dz(\sz,U)>0$ such that for all $x,y\in U$ and $\dz<\dz(\sz,U)$,  we have
 $$\frac{\dz}{\sigma}\mathcal{L}^{\dz}_{\sigma}(x,y)<\ln \sqrt {2}$$
and hence, by Lemma \ref{distance0},
\begin{equation*}
\mathscr C^H_\sigma (y-x)\leq e^{\frac{4\dz}{\sigma}\mathcal{L}^{\dz}_{\sigma}(x,y)}\mathcal{L}^{\dz}_{\sigma}(x,y)
\le 4\mathcal{L}^{\dz}_{\sigma}(x,y).
\end{equation*}
By \eqref{claim0}, for all  $\sz>8\Lambda \|u\|^2_{C^{0,1}(\overline U)} $ and $\dz<\dz(\sz,U)$,  we have
\begin{equation}\label{claim}
|u(y)-u(x )|\le  \mathcal{L}^{\dz}_{\sigma}(x,y) ,\quad
\forall  x,y\in  U.
\end{equation}

Note that
$$\mathscr{A}_H (u^\ez)+\ez \Delta u^\ez\ge 0\quad \mbox{in $U $} $$
in viscosity sense  and by Lemma \ref{com-pr1},
$$\mathscr{A}_H (\mathcal{ L}^\dz_\sz(x_0,x))+\ez \Delta \mathcal{ L}^\dz_\sz(x_0,x)\le -\frac{\dz\sz}{2}+\ez n \mu_1\quad \mbox{in $U $}$$
in viscosity sense.
For all  $\sz>8\Lambda \|u\|^2_{C^{0,1}(\overline U)} $ and $\dz<\dz(\sz,U)$
 and if $0<\ez<\frac{\dz\sz}{2n\mu_1}$,
by Lemma \ref{com-pr} we have
\begin{equation}\label{holder2}
u^\ez(x)-u(x_0)\le \mathcal{ L}^\dz_\sz(x_0,x),\quad
\forall x\in U,x_0\in\partial U.
\end{equation}
By similar argument, for all  $\sz>8\Lambda \|u\|^2_{C^{0,1}(\overline U)} $ and $\dz<\dz(\sz,U)$,
if    $0<\ez<\frac{\dz\sz}{2n\mu_2}$,
 we have
$$u^\ez(x)-u(x_0)\ge -\mathcal{  L}^\dz_\sz(x,x_0),\quad
\forall x\in \partial U,x_0\in\partial U.$$

 We therefore conclude that  for    $\sz=8\Lambda \|u\|^2_{C^{0,1}(\overline U)} $ and $\dz<\dz(\sz,U)$
 if $0<\ez<\min\{\frac{\dz\sz}{2n\mu_1},\frac{\dz\sz}{2n\mu_2}\} $,
$$
|u^\ez(x)-u(x_0)|\le  \mathcal{C}^H _\sz(x_0,x),
\quad \forall x\in U,x_0\in\partial U.
$$
Thus, there exist  $\ez_\ast$ and $C$ depending on $U$, $\|u\|_{C^{0,1}(\overline U)}$,  $H$, $\dz,\sz$  such that
for all  $0<\ez<\ez_\ast$, we have
$$
|u^\ez(x)-u(x_0)|\le   C|x-x_0|
\quad \forall x\in U,x_0\in\partial U.
$$
The proof of Theorem 2.2 is complete. \end{proof}

\section{A generalization of Evans-Smart' adjoint method}

Let $H$, $H^\gz$, $u$  $u^\gz$   and $u^{\gz,\ez}$  be as in Section 2.
For convenience, we write $H^\gz$ as $H$, and $u^\gz$ as $u$, $u^{\gz,\ez}$ as $u^\ez$ below.
%For each $\ez\in(0,1]$, denote by
%$u^\ez$ the smooth solution to \eqref{req}.
Let $L_\ez$ be the linearized operator given in \eqref{lin-op},   %that is,
%\begin{equation}\label{oper1}
%L_\ez(v):=-H_{p_i}(Du^\ez)H_{p_j}(Du^\ez)v_{x_ix_j}-2
%H_{p_ip_l}(Du^\ez)u^\ez_{x_ix_j}H_{p_j}(Du^\ez)v_{x_l}-\ez \Delta v
%\end{equation}
%for any $v\in C^\fz (U)$.
and  $L^\ast_\ez$  be its  dual operator, that is,
\begin{align}\label{oper2}
L_\ez^\star(v)&:=-[H_{p_i}(Du^\ez)H_{p_j}(Du^\ez)v]_{x_ix_j}+
2[H_{p_ip_l}(Du^\ez)u^\ez_{x_ix_j}H_{p_j}(Du^\ez)v]_{x_l}-\ez\Delta v
%&=- H_{p_i}(Du^\ez)H_{p_j}(Du^\ez)v _{x_ix_j}-\ez\Delta v\nonumber \\
%&\quad
% -\left[ H_{p_ip_s}(Du^\ez)u_{x_sx_i}H_{p_jp_t}(Du^\ez)u_{x_tx_j} ^\ez -H_{p_ip_s}(Du^\ez)u_{x_sx_t}^\ez H_{p_tp_j}(Du^\ez)u_{x_jx_i}^\ez\right] v\nonumber
 %\\
 %&=- H_{p_i}(Du^\ez)H_{p_j}(Du^\ez)v _{x_ix_j}-\ez\Delta v
% -[\det D^2u^\ez]   [\det D^2_{pp} H(Du^\ez) ]v\nonumber
\end{align}
for any $v\in C^\fz (U)$.
Observe that
$$
\int_\rn L_\ez(v)(x)w(x)\,dx= \int_\rn v(x)L^\ast_\ez(w)(x)\,dx\quad\forall v, w\in C^\fz_c (U).$$

Fix a smooth domain $V\Subset U$.
For each point $x_0\in V$, we consider the adjoint problem
 \begin{equation}\label{adj}
 L^\ast_\ez(v)  =\delta_{x_0}  \  {\rm in}\  V; v  =0  \  {\rm on}\ \partial V,
 \end{equation}
where $\delta_{x_0}$ denotes the Dirac measure at $x_0$. Equivalently,
$$
\int_{V} v(x)L_\ez(\phi)(x)\,dx=\phi(x_0)\quad \forall \phi\in C^\fz_0(V);\
 v  =0  \  {\rm on}\ \partial V.$$
 Then we have the following result.
\begin{thm}\label{ex-ad}
For each point $x_0\in V$,
there exists a unique solution $\Theta^\ez\in C^\fz(\overline V\setminus\{x_0\})$
of the linear adjoint problem \eqref{adj} such that $\Theta^\ez\ge 0$ in $V$.
\end{thm}
\begin{proof}
Consider problem
 \begin{equation*}
L_\ez (w) =0 \quad {\rm in}\  V;
w =0 \quad {\rm on}\ \partial V.
\end{equation*}
By  Theorem \ref{ex-una}, there exists a unique solution
$\omega\equiv 0$ on $\overline V$. So that $0$ is not an eigenvalue of the operator $L_\ez$,
and hence $0$ is   not an eigenvalue of $L_\ez^\star$.
Applying standard linear elliptic PDE theory, there exists smooth
Green's function $\Theta^\ez\in C^\fz(\overline B(0,2)\backslash\{x_0\})$.
Next we show that $\Theta^\ez\ge 0$. For any $f\in C^\fz (V)$ and
$f\ge 0$ in $V$, we introduce  the solution $\omega^\ez$ of the linear boundary value problem
 \begin{equation}\label{eq0}
L_\ez (\omega^\ez) =f \quad {\rm in}\  V;
\omega^\ez =0 \quad {\rm on}\ \partial V.
\end{equation}
By Theorem \ref{ex-una}, we know that
there exists a unique solution $0\le \omega^\ez\in C^\fz(V)$.
Multiply the equation in \eqref{eq0} by $\Theta^\ez$, we have
\begin{align*}
&\int_Vf\Theta^\ez\,d x=\int_VL_\ez(\omega^\ez )\Theta^\ez\,d x\\
&=\int_V [-H_{p_i}(Du^\ez)H_{p_j}(Du^\ez)\omega^\ez_{x_ix_j}-2H_{p_j}(Du^\ez)
H_{p_ip_l}(Du^\ez)u^\ez_{x_ix_j}\omega^\ez_{x_l}-\ez \omega^\ez_{x_ix_i}]\Theta^\ez\,d x.
\end{align*}
By integration by parts, $\omega^\ez|_{\partial V}=0$ and $\Theta^\ez|_{\partial V}=0$, we have
\begin{align*}\
&\int_V -H_{p_i}(Du^\ez)H_{p_j}(Du^\ez) \omega^\ez_{x_ix_j}\Theta^\ez\,dx\\
&\quad=
\int_V -(H_{p_i}(Du^\ez)H_{p_j}(Du^\ez)\Theta^\ez)_{x_ix_j}\omega^\ez\,dx -\int_{\partial V}H_{p_i}(Du^\ez)H_{p_j}(Du^\ez)\Theta^\ez\omega^\ez_{x_j} \cos(\overrightarrow{N},x_i)\,ds\\
&\quad\quad+\int_{\partial V} (H_{p_i}(Du^\ez)H_{p_j}(Du^\ez)\Theta^\ez)_{x_i}\omega^\ez \cos(\overrightarrow{N},x_j)\,ds\\
&\quad=\int_V -(H_{p_i}(Du^\ez)H_{p_j}(Du^\ez)\Theta^\ez)_{x_ix_j}\omega^\ez\,dx,
\end{align*}
where $\overrightarrow{N}$ denotes the outward pointing unit normal along $\partial V$.
By similar calculation, which lead to
\begin{align*}
&\int_Vf\Theta^\ez\,d x=\int_VL_\ez^\star(\Theta^\ez)\omega^\ez \,d x
=\omega^\ez(x_0)\ge 0.
\end{align*}
Since for all $f\ge 0$ holds, that is $\Theta^\ez\ge 0$.
\end{proof}

\begin{lem} Denote by
$\overrightarrow{N}$ denotes the outward pointing unit normal along $\partial V$. Then
 $$\cos(\overrightarrow{N},x_i)=-\frac{\Theta^\ez_{x_i}}{|D\Theta^\ez|}\quad\forall i=1,\cdots,n.$$
\end{lem}
 We have the following  connection of between operator $L_\ez$ and $\Theta_\ez$.
\begin{lem}\label{L-L}
For any $v\in C^\fz(\overline V)$, we have
\begin{align*}
\int_V L_\ez (v)\Theta^\ez\,d x+\int_{\partial V}v\rho^\ez\,ds=v(x_0),
\end{align*}
where dentes
$$\rho^\ez:=\frac{\langle D_p H(Du^\ez),D\Theta^\ez\rangle^2}{|D\Theta^\ez|}+\ez|D\Theta^\ez|.$$

\end{lem}
\begin{proof}
By  integrate by parts and $\Theta^\ez|_{\partial V}=0$, we have
\begin{align*}
\int_V& L_\ez (v)\Theta^\ez\,d x\\
&=
\int_V [-H_{p_i}(Du^\ez)H_{p_j}(Du^\ez)v_{x_ix_j}-2H_{p_j}(Du^\ez)
H_{p_ip_l}(Du^\ez)u^\ez_{x_ix_j}v_{x_l}-\ez v_{x_ix_i}]\Theta^\ez\,d x\\
&=\int_V [-(H_{p_i}(Du^\ez)H_{p_j}(Du^\ez)\Theta^\ez)_{x_ix_j}+2(H_{p_j}(Du^\ez)H_{p_ip_l}
(Du^\ez)u^\ez_{x_ix_j}\Theta^\ez)_{x_l} -\ez (\Theta^\ez)_{x_ix_i}]v\,d x\\
&\quad+\int_{\partial V}v[(H_{p_i}(Du^\ez)H_{p_j}(Du^\ez)\Theta^\ez)_{x_i}+\ez(\Theta^\ez)_{x_j}]\cos(\overrightarrow{N},x_j) \,ds,
\end{align*}
where $\overrightarrow{N}$ denotes the outward pointing unit normal along $\partial V$.
Note that $\cos(N,x_i)=-\Theta^\ez_{x_i}/|D\Theta^\ez|$ and $\Theta^\ez|_{\partial V}=0$, we have
\begin{align*}
\int_V L_\ez (v)\Theta^\ez\,d x&=
\int_V L_\ez^\star(\Theta^\ez)v\,d x-\int_{\partial V}v\left[\frac{\langle D_p H(Du^\ez),D\Theta^\ez\rangle^2}{|D\Theta^\ez|}+\ez |D\Theta^\ez|\right]\,ds.
\end{align*}
Denote
$$\rho^\ez:=\frac{\langle D_p H(Du^\ez),D\Theta^\ez\rangle^2}{|D\Theta^\ez|}+\ez |D\Theta^\ez|,$$
this complete proof of the Lemma \ref{L-L}.
\end{proof}

Since $L_\ez(1)=0$ in $V$, the following follows from Lemma \ref{L-L} obviously.

\begin{cor}\label{es00} We have
 $$
\int_{\partial V}\rho^\ez\,ds=1.
$$
\end{cor}
Letting $v=H(Du^\ez)$ in Lemma \ref{L-L}, we also have the following.
\begin{lem}\label{es0} We have
\begin{align*}
 &\int_V [\langle D^2_{pp}H(Du^\ez) D[H(Du^\ez)],D[H(Du^\ez)]\rangle+\ez  H_{p_kp_s}(Du^\ez)u^\ez_{x_kx_i}u^\ez_{x_sx_i}]\Theta^\ez\,d x\le  \|H(Du^\ez)\|_{L^\fz(V)}.
\end{align*}
\end{lem}

\begin{proof}
By Lemma \ref{L-L},
$$-\int_V L_\ez(H(Du^\ez))\Theta^\ez\,dx
 =\int_{\partial V}H(Du^\ez)\rho^\ez\,ds-H(Du^\ez(x_0)) \le  \|H(Du^\ez)\|_{L^\fz(V)}.$$
Write
\begin{align*}
L_\ez(H(Du^\ez))%&=-H_{p_i}(Du^\ez)H_{p_j}(Du^\ez)[H(Du^\ez)]_{x_ix_j}\\
%&\quad-2H_{p_j}(Du^\ez)H_{p_ip_l}(Du^\ez)
%u^\ez_{x_ix_j}[H(Du^\ez)]_{x_l}-\ez [H(Du^\ez)]_{x_ix_i}\\
&=-H_{p_i}(Du^\ez)H_{p_j}(Du^\ez)[H_{p_kp_s}(Du^\ez)u^\ez_{x_kx_i}
u^\ez_{x_sx_j}+H_{p_k}(Du^\ez)u^\ez_{x_kx_ix_j}]\\
&\quad -2H_{p_j}(Du^\ez)H_{p_ip_l}(Du^\ez)
u^\ez_{x_ix_j}H_{p_k}(Du^\ez)u^\ez_{x_kx_l}\\
&\quad-\ez [H_{p_kp_s}(Du^\ez)u^\ez_{x_kx_i}
u^\ez_{x_sx_i}+H_{p_k}(Du^\ez)u^\ez_{x_kx_ix_i}].
\end{align*}
Since $ L_\ez(u^\ez_{x_s})=0$, we have
%\begin{align*}0=L_\ez[u^\ez_s]&=-H_{p_i}(Du^\ez)H_{p_j}(Du^\ez) H_{p_k}(Du^\ez)u^\ez_{x_kx_ix_j}\\
%&\quad-2H_{p_j}(Du^\ez)H_{p_ip_l}(Du^\ez)
%u^\ez_{x_ix_j}H_{p_k}(Du^\ez)u^\ez_{x_kx_l}-\ez H_{p_k}(Du^\ez)u^\ez_{x_kx_ix_i}
%\end{align*} we have
\begin{align}\label{EQ3.2}
L_\ez(H(Du^\ez))&=-H_{p_kp_s}(Du^\ez)[H_{p_i}(Du^\ez)H_{p_j}(Du^\ez)u^\ez_{x_kx_i}u^\ez_{x_sx_j}
+\ez u^\ez_{x_kx_i}
u^\ez_{x_sx_i}]\\
&=-\langle D^2_{pp}H(Du^\ez) D[H(Du^\ez)],D[H(Du^\ez)]\rangle-\ez  H_{p_kp_s}(Du^\ez)u^\ez_{x_kx_i}u^\ez_{x_sx_i} \nonumber
\end{align}
as desired.
\end{proof}

We further need   an exponential estimate.
\begin{lem}\label{es7}
Moreover, for all $0<\mu<\frac{\lambda }{8n}$ we have
 \begin{align*}
&\int_{\partial V}\ez e^{\frac{\mu}{\ez}[H(Du^\ez(x_0))-H(Du^\ez)]}\rho^\ez\,ds\\
&\quad\quad+\mu\int_Ve^{\frac{\mu}{\ez}[H(Du^\ez(x_0))-H(Du^\ez)]}
  \langle D^2_{pp}H(Du^\ez)D[H(Du^\ez)],D[H(Du^\ez)]\rangle \Theta^\ez\,dx\\
 &\quad\quad+\ez\mu \int_Ve^{\frac{\mu}{\ez}[H(Du^\ez(x_0))-H(Du^\ez)]}  H_{p_kp_s}(Du^\ez)u^\ez_{x_sx_i}u^\ez_{x_kx_i} \Theta^\ez\,dx\\
 &\quad\le 2\ez.
\end{align*}

\end{lem}

 \begin{proof}
Let
$$\phi(r)=\ez e^{\frac{\mu}{\ez}[\alpha_\ez-r]}
\quad {\rm and} \ \alpha_\ez :=H(Du^\ez(x_0)),$$
where $0<\mu\le \frac{\lambda }{8n}$.
Similarly to \eqref{EQ3.2} we have
\begin{align*}
L_\ez((\phi\circ H)(Du^\ez))&=-(\phi\circ H)_{p_kp_s}(Du^\ez)[H_{p_i}(Du^\ez)H_{p_j}(Du^\ez)u^\ez_{x_kx_i}u^\ez_{x_sx_j}
+\ez u^\ez_{x_kx_i}
u^\ez_{x_sx_i}]\\
&=-\langle D^2_{pp}H(Du^\ez) D[H(Du^\ez)],D[H(Du^\ez)]\rangle-\ez  (\phi\circ H)_{p_kp_s}(Du^\ez)u^\ez_{x_kx_i}u^\ez_{x_sx_i}.
\end{align*}
Since
$$(\phi\circ H)_{p_kp_s}=(\phi''\circ H)H_{p_k}H_{p_s}+ (\phi'\circ H)H_{p_kp_s} $$
 and $\mathscr A_H[u^\ez]=-\ez\Delta u^\ez$, we get
\begin{align*}&L_\ez((\phi\circ H) (Du^\ez) )\\
&=[(\phi''\circ H)(Du^\ez)H_{p_k}(Du^\ez)H_{p_s}(Du^\ez)+ (\phi'\circ H)(Du^\ez)H_{p_kp_s}(Du^\ez)]\\
&\quad\times[H_{p_i}(Du^\ez)H_{p_j}(Du^\ez)u^\ez_{x_kx_i}u^\ez_{x_sx_j}
+\ez u^\ez_{x_kx_i}
u^\ez_{x_sx_i}] \\
&=  - (\phi' \circ H)(Du^\ez)\langle D^2_{pp}H(Du^\ez)D[H(Du^\ez)],D[H(Du^\ez)]\rangle
-(\phi'' \circ H)(Du^\ez)   \ez |D H(Du^\ez)|^2\\
&\quad -\ez (\phi' \circ H)(Du^\ez) H_{p_kp_s}(Du^\ez)u^\ez_{x_sx_i}u^\ez_{x_kx_i}- \phi''(H(Du^\ez))  \ez^2(\Delta u^\ez )^2.
%&= -(\phi''\circ H)(Du^\ez) [(\mathscr A_H[u^\ez])^2+ \ez |D H(Du^\ez)|^2]\\
%&\quad- (\phi' \circ H)(Du^\ez)\left[\langle D^2_{pp}H(Du^\ez)D[H(Du^\ez)],D[H(Du^\ez)]\rangle+
%\ez H_{p_kp_s}(Du^\ez)u^\ez_{x_sx_i}u^\ez_{x_kx_i}\right].
\end{align*}
%Since $\mathscr A_H[u^\ez]=-\ez\Delta u^\ez$, one gets
%\begin{align*}&L_\ez((\phi\circ H)(Du^\ez) ) \\
%&=  - (\phi' \circ H)(Du^\ez)\langle D^2_{pp}H(Du^\ez)D[H(Du^\ez)],D[H(Du^\ez)]\rangle
%-(\phi'' \circ H)(Du^\ez)   \ez |D H(Du^\ez)|^2\\
%&\quad -\ez (\phi' \circ H)(Du^\ez) H_{p_kp_s}(Du^\ez)u^\ez_{x_sx_i}u^\ez_{x_kx_i}- \phi''(H(Du^\ez))  \ez^2(\Delta u^\ez )^2.
%\end{align*}

Note that
$$\phi'(r)=  -  {\mu}  e^{\frac{\mu}{\ez}[\alpha_\ez-r]}\le0, \phi''(r)= \frac {\mu^2}{\ez} e^{\frac{\mu}{\ez}[\alpha_\ez-r]}\ge0.$$
Since the strongly convexity of $H$ implies $$|p|^2\le\frac1\lz \langle D^2_{pp}H(Du^\ez)p,p\rangle,$$
by $[1-     \frac{n\mu}  \lambda]\ge 1/2$ we have
\begin{align}\label{EQ3.3}L_\ez((\phi\circ H)(Du^\ez) )
&\ge   \frac {\mu}2 e^{\frac{\mu}{\ez}[\alpha_\ez-H(Du^\ez)]}  \langle D^2_{pp}H(Du^\ez)D[H(Du^\ez)],D[H(Du^\ez)]\rangle
\\
&\quad +\ez\frac\mu2   e^{\frac{\mu}{\ez}[\alpha_\ez-H(Du^\ez)]}   H_{p_kp_s}(Du^\ez)u^\ez_{x_sx_i}u^\ez_{x_kx_i}\nonumber.
\end{align}

Since Lemma \ref{L-L} implies
$$\int_{\partial V} (\phi\circ H)(Du^\ez) \rho^\ez\,ds+
 \int_{V} L_\ez((\phi\circ H)(Du^\ez) )\Theta^\ez\,dx=(\phi\circ H)(Du^\ez(x_0)) = \ez.$$
  We obtain the desired estimate.
\end{proof}

Applying Lemma \ref{es7}, we will get the following upper bound.
\begin{lem}\label{es10}
We have
\begin{align*}
&\int_V[H(Du^\ez)]^2\Theta^\ez\,dx\le  C( \lambda,\Lambda,\eta)[1+\|u^\ez\|_{L^\fz(V)}]\left[1+  \|H(Du^\ez)\|_{L^\fz(V)}\right]+  C( \lambda,\Lambda)\eta^2 \int_V\Theta^\ez\,dx.
\end{align*}

\end{lem}
\begin{proof}

By
  $\mathscr{A}_H (u^\ez)+\ez \Delta u^\ez=0$, a direct calculation implies that
\begin{align*}
L_\ez(\frac12(u^\ez)^2)%&=-H_{p_i}(Du^\ez)H_{p_j}(Du^\ez)[\psi(u^\ez)]_{x_ix_j}-2H_{p_j}(Du^\ez)
%H_{p_ip_l}(Du^\ez)u^\ez_{x_ix_j}[\psi(u^\ez)]_{x_l}-\ez [\psi(u^\ez)]_{x_ix_i}\\
%&=-H_{p_i}(Du^\ez)H_{p_j}(Du^\ez)[ u^\ez_{x_i}u^\ez_{x_j}+ u^\ez u^\ez_{x_ix_j}]\\
%&\quad-2 u^\ez  H_{p_j}(Du^\ez)
%H_{p_ip_l}(Du^\ez)u^\ez_{x_ix_j} u^\ez _{x_l} -\ez [ u^\ez_{x_i}u^\ez_{x_i}+ u^\ez u^\ez_{x_ix_i}]\\
&=- [\langle D_p H(Du^\ez),Du^\ez\rangle^2+\ez|Du^\ez|^2]-2 u^\ez\langle D^2_{pp}H(Du^\ez)D[H(Du^\ez)], Du^\ez \rangle.
\end{align*}

By the    convexity of $H$ and $H(0)=0$, we have
$$\langle D_p H(Du^\ez),Du^\ez\rangle^2\ge [H(Du^\ez) ]^2.$$
By the Young's inequality, we have
\begin{align*}
&2 u^\ez\langle D^2_{pp}H(Du^\ez)D[H(Du^\ez)], Du^\ez \rangle\\
&\quad\le  C(\eta)
 |u^\ez|^2\langle D^2_{pp}H(Du^\ez)D[H(Du^\ez)],D[H(Du^\ez)]\rangle+ \eta \langle D^2_{pp}H(Du^\ez) Du^\ez , Du^\ez \rangle.
\end{align*}
By the  strongly concavity of $H$, we have
\begin{align*}
 \eta\langle D^2_{pp}H(Du^\ez) Du^\ez , Du^\ez \rangle
&\le \eta
\frac{2\Lambda}{\lambda} H(Du^\ez)\le \frac12[H(Du^\ez)]^2 +C(\lambda,\Lambda)\eta^2.
\end{align*}
Thus
\begin{align*} \int_V[H(Du^\ez) ]^2\Theta^\ez\,dx
&\le
-\int_V L_\ez(\frac12(u^\ez)^2)\Theta^\ez\,dx+ \frac12\int_V[H(Du^\ez)]^2\Theta^\ez\,dx+ C(\lambda,\Lambda)\eta\int_V \Theta^\ez\,dx\\
&\quad+
C(\eta)\int_V|u^\ez|^2\langle D^2_{pp}H(Du^\ez)D[H(Du^\ez)],D[H(Du^\ez)]\rangle\Theta^\ez\,dx
\end{align*}
By Lemma \ref{L-L},
$$-\int_V L_\ez(\frac12(u^\ez)^2)\Theta^\ez\,dx\le 2\|u^\ez\|_{L^\fz(V)},$$
and hence,
\begin{align*}\int_V[H(Du^\ez) ]^2\Theta^\ez\,dx
&\le C(\eta,\lambda,\Lambda)[1+\|u^\ez\|_{L^\fz(V)}^2][1+  \|H(Du^\ez)\|_{L^\fz(V)}]+ C(\lambda,\Lambda)\eta^2\int_V \Theta^\ez\,dx.
\end{align*}

\end{proof}

Moreover, we also need an integral estimate of $\Theta^\ez$.
\begin{lem}\label{exp} Let $x_0\in V$ and $\alpha_\ez:= H(Du^\ez(x_0))>0$.

(i)For any $0<\mu<\frac{\lambda }{8n}$ and $0<\bz< H(Du^\ez(x_0))$,  we have
\begin{align*}
\int_{V\cap \{H\le \beta\}}\Theta^\ez\,dx\le C(\lambda,\Lambda)\frac{1}{\ez}e^{\frac{\mu[\beta-H(Du^\ez(x_0))]}{\ez}}.
\end{align*}

  (ii) If $\liminf_{\ez\to 0} H(Du^\ez(x_0))\ge \alpha>0$, we have
 $$ \int_{V}\Theta^\ez\,dx\le C(\lambda,\Lambda)\frac1\ez  e^{-\frac{\mu \alpha  }{2\ez}}++C( \lambda,\Lambda,\|u\|_{L^\fz(V)},{\rm dist}(V,\partial U))\frac1{\alpha^2}.$$

\end{lem}

\begin{proof}
For each $0<\mu<\frac{\lambda }{8n}$, define
$$\phi(r)=\ez e^{\frac{\mu}{\ez}(\alpha_\ez-r)} $$
and set  $v(x)=(\phi\circ H)(Du^\ez(x)) |x|^2$.  By Lemma \ref{L-L},
$$
 \int_V L_\ez(v)\Theta^\ez\,dx=    v(x_0)- \int_{\partial V}v\rho^\ez\,ds.   $$
Then
\begin{align*}
L_\ez(v)&=(\phi\circ H)(Du^\ez)  L_\ez(|x|^2)+|x|^2L_\ez((\phi\circ H)(Du^\ez))\\
&\quad-
2\langle D_p H(Du^\ez),D(\phi\circ H)(Du^\ez)\rangle\langle D_p H(Du^\ez),D|x|^2\rangle-2\ez \langle D(\phi\circ H)(Du^\ez),D|x|^2\rangle.
\end{align*}
Write $$K:=   e^{\frac{\mu}{\ez}[\alpha_\ez-H(Du^\ez)]}
[\langle D^2_{pp}H(Du^\ez)D[H(Du^\ez)],D[H(Du^\ez)]\rangle+ \ez H_{p_kp_s}(Du^\ez)u^\ez_{x_sx_i}u^\ez_{x_kx_i}].$$
%Then $\mu\int_V K\Theta^\ez\,dx\le C(n,\lambda)\ez$.
By \eqref{EQ3.3}, we have
\begin{align*}|x|^2 L_\ez((\phi\circ H)(Du^\ez) )
&\le    4  \mu K.
\end{align*}

Note that
\begin{align*}
L_\ez(|x|^2)&=-H_{p_i}(Du^\ez)H_{p_j}(Du^\ez)(|x|^2)_{x_ix_j}-2H_{p_j}(Du^\ez)
H_{p_ip_l}(Du^\ez)u^\ez_{x_ix_j}(|x|^2)_{x_l}-\ez(|x|^2)_{x_ix_i}\\
&=-2[|D_p H(Du^\ez)|^2+\ez n]-2\langle D^2_{pp} H(Du^\ez) D[H(Du^\ez)],D|x|^2\rangle%\\
%&=-2[|D_p H(Du^\ez)|^2+\ez n]-\frac18\langle D^2_{pp} H(Du^\ez) D[H(Du^\ez)],D[H(Du^\ez)]\rangle - C\Lambda
\end{align*}
and hence by the Young's inequality,
\begin{align*}
 (\phi\circ H)(Du^\ez)L_\ez(|x|^2)
&\le -2(\phi\circ H)(Du^\ez) [ |D_pH(Du^\ez)|^2 +\ez n]+ C(\lambda,\Lambda) K + \frac\ez 8  (\phi\circ H)(Du^\ez)
\end{align*}

Using  $\langle D_p H(Du^\ez),D [H(Du^\ez)] \rangle=\mathscr A_H[u^\ez]=-\ez\Delta u^\ez  $ we also have
\begin{align*}
&-
2\langle D_p H(Du^\ez),D(\phi\circ H)(Du^\ez)\rangle\langle D_p H(Du^\ez),D|x|^2\rangle-2\ez \langle D(\phi\circ H)(Du^\ez),D|x|^2\rangle\\
&\quad= -
2(\phi'\circ H)(Du^\ez) [\langle D_p H(Du^\ez),D [H(Du^\ez)] \rangle\langle D_p H(Du^\ez),D|x|^2\rangle+\ez \langle D [H(Du^\ez)] ,D|x|^2\rangle]\\
&\quad= -
2(\phi'\circ H)(Du^\ez)[ -\ez\Delta u^\ez\langle D_p H(Du^\ez),D|x|^2\rangle+\ez \langle D [H(Du^\ez)] ,D|x|^2\rangle,
\end{align*}
by
$\phi'(r)=  -  {\mu}  e^{\frac{\mu}{\ez}[\alpha_\ez-H(Du^\ez)]}\le0$ and Young's inequality,  which is bounded  by
\begin{align*}
 C(\lambda,\Lambda) \mu K +\frac18(\phi\circ H)(Du^\ez) |D_pH(Du^\ez)|^2 + \frac\ez 8  (\phi\circ H)(Du^\ez).
\end{align*}

Thus $$L_\ez(v)\le - (\phi\circ H)(Du^\ez) [ |D_pH(Du^\ez)|^2 +\ez n]+ C(\lambda,\Lambda) \mu K.$$

Therefore, applying Lemma \ref{es7}  we get
\begin{align*}
&  \int_V  (\phi\circ H)(Du^\ez) [ |D_pH(Du^\ez)|^2 +\ez n]\Theta^\ez\,dx\\
&\quad  \le 4\int_{\partial V}(\phi\circ H)(Du^\ez)  \rho^\ez\,ds
 +C(\lambda,\Lambda)\int_V \mu  K\Theta^\ez\,dx   \le C(\lambda,\Lambda)\ez \nonumber.
\end{align*}
We conclude that
\begin{align*}
&\ez^2\int_V e^{\frac{\mu[\alpha_\ez-H(Du^\ez)]}{\ez}} \Theta^\ez\,dx\le
C(\lambda,\Lambda) \ez ,
\end{align*}
 and hence
\begin{align}\label{exp3}
\int_{V\cap \{H(Du^\ez)\le \beta\}}\Theta^\ez\,dx\le C(\lambda,\Lambda)\frac1\ez  e^{-\frac{\mu[\alpha_\ez-\beta]}{\ez}}.
\end{align}
This implies that
\begin{align*}
\int_{V}\Theta^\ez\,dx&=\int_{V\cap \{H(Du^\ez)\le \frac{\alpha}{2}\}}\Theta^\ez\,dx+
\int_{V\cap \{H(Du^\ez)> \frac{\alpha}{2}\}}\Theta^\ez\,dx\\
&\le  C(\lambda,\Lambda)\frac1\ez  e^{-\frac{\mu[2\alpha_\ez-\az ]}{2\ez}}+
  \frac4{\alpha^2}\int_{V }  [H(Du^\ez)]^2  \Theta^\ez\,dx.
\end{align*}

Let $C_0(\lambda,\Lambda)\eta^2 \frac4{\alpha^2}=\frac12 $, that is, $\eta^2=  \frac {\alpha^2} {8C_0(\lambda,\Lambda)}$. Apply Lemma \ref{es10}, Theorem \ref{ex-una} and $\alpha_\ez\ge \alpha$ for all
$\ez\in (0,1]$, we have
$$ \int_{V}\Theta^\ez\,dx\le C(\lambda,\Lambda)\frac1\ez  e^{-\frac{\mu\az }{2\ez}}+C( \lambda,\Lambda,\|u\|_{L^\fz(V)},{\rm dist}(V,\partial U))\frac1{\alpha^2}.$$
\end{proof}
\section{Proof of Theorem \ref{Fla}}

Let $U=B(0,3)$ and $V=B(0,2)$ in this section. Let $H$, $H^\gz$, $u$  $u^\gz$   and $u^{\gz,\ez}$  be as in Section 2.
For convenience, we write $H^\gz$ as $H$, and $u^\gz$ as $u$, $u^{\gz,\ez}$ as $u^\ez$ below.

%For simplicity, we write $H^\gz$ as $H$, and $u^\gz$ as $u$, $u^{\gz,\ez}$ as $u^\ez$.
  %We will built an flatness estimate.
%\begin{thm}\label{Fla}
%Suppose that
%\begin{equation}\label{cond1}
%\max_{B(0,3)}|u^\ez-e_n\cdot x|\le \tau
%\end{equation}
%for some $0<\tau<1$  and that
%\begin{equation}\label{H(Du)1}
% H(Du^\ez(x_0))\ge H(e_n)- \delta
%\end{equation}
%for some $0<\dz< H(e_n)/2 $ and $x_0\in B(0,1)$.
%  Then
%\begin{align}\label{flat-ex}
%|Du^\ez(x_0)-e_n|^2&\le C(\lambda,\Lambda)\left[\tau+\dz+\frac1{\ez}e^{-\frac{\mu\delta}{\ez}} \right].
% \end{align}
% where $\mu=\frac\lz {16n}$.
%\end{thm}
%\begin{lem}\label{flat}
%If
% we have
%\begin{align}\label{fl0}
%\int_{B(0,2)}\langle D_pH(Du^\ez),Du^\ez-e_n\rangle^2\Theta^\ez\,dx
%\le C\frac{\Lambda [(\Lambda )^4+1]^6}{\lambda ^{12}}\tau .
%\end{align}
%%Moreover, for any $\dz>0$ we have
%%\begin{align}\label{fl1}
%%\int_{B(0,2)\cap \{H(Du^\ez)-H(e_n)\ge 2\dz\}}\Theta^\ez\,dx
%%\le C\frac{\Lambda ^{\frac 3 2}[\Lambda ^4+1]^6}
%%{\lambda ^{\frac{15} {2}}}\frac{\lambda^{\frac 1 2}}{\dz }.
%%\end{align}
%
%\end{lem}
 Note that the condition \eqref{cond1} and Theorem \ref{ex-una} implies that  $$\sup_{U}|u^\ez|\le 4\quad\mbox{ and  } \sup_{V}|Du^\ez|\le C(\lambda,\Lambda).$$
 Moreover, let $L_\ez$ and $\Theta^\ez$ is given in Theorem \ref{ex-ad}. The condition \eqref{H(Du)1} implies that Lemma \ref{exp} (ii) holds, that is $$\int_V\Theta^\ez\,dx\le C(\lambda,\Lambda).$$
  The proof of Theorem \ref{Fla} is then divided into 3 steps.

\medskip
\noindent {\it Step 1.} We first show that
$$\int_{V}[H(e_n)-H(Du^\ez)]_+\Theta^\ez\,dx\le C(\lambda,\Lambda)[\dz+\frac1\ez  e^{ -\frac{\mu}{\ez}\dz  }].$$
Here and below $f_+=\max\{f,0\}$.
Observe that
 \begin{align*}
\int_{V}[H(e_n)-H(Du^\ez)]_+\Theta^\ez\,dx &=\int_{V\cap \{H(Du^\ez)\le H(e_n)-2\delta\}}[H(e_n)-H(Du^\ez)]\Theta^\ez\,dx\\
   &\quad+\int_{V\cap \{H(e_n)-2\delta\le H(Du^\ez)\le H(e_n) \}}[H(e_n)-H(Du^\ez)]\Theta^\ez\,dx.
 \end{align*}
By Lemma \ref{exp} (i), we have
\begin{align*}
 \int_{V\cap \{H(Du^\ez)\le H(e_n)-2\delta\}}[H(e_n)-H(Du^\ez)]\Theta^\ez\,dx\le   H(e_n) \frac1\ez  e^{ \frac{\mu}{\ez} [H(e_n)-2\delta- H(Du^\ez(x_0))] }\le \Lambda\frac1\ez  e^{ -\frac{\mu}{\ez}\dz  }.
 \end{align*}
By Lemma \ref{exp} (ii), we also have
$$\int_{V\cap \{H(e_n)-2\delta\le H(Du^\ez)\le H(e_n) \}}[H(e_n)-H(Du^\ez)]\Theta^\ez\,dx\le  2\dz \int_{V } \Theta^\ez\,dx\le C(\lambda,\Lambda)\dz.$$

\medskip
\noindent
{\it Step 2.}
We show that
\begin{align*}
\int_{V}\langle D_p H(Du^\ez),Du^\ez-e_n\rangle^2\Theta^\ez\,dx
&\le   C(\lambda,\Lambda) \tau[1+\frac1\ez  e^{ -\frac{\mu}{\ez}\dz  }].
\end{align*}

Taking $v=(u^\ez-x_n)^2$ in Lemma \ref{L-L}, we have
$$
(u^\ez(x_0)-x_n^0)^2=\int_{B(0,2)}L_\ez ((u^\ez-x_n)^2)\Theta^\ez\,dx
+\int_{\partial B(0,2)}(u^\ez-x_n)^2\rho\,ds
$$
and hence, by \eqref{cond1},
\begin{equation}\label{EQ4.1}0\le \int_{B(0,2)}L_\ez ((u^\ez-x_n)^2)\Theta^\ez\,dx+\tau^2.
\end{equation}
Since $\mathscr{A}_H (u^\ez)+\ez \Delta u^\ez=0$, one has
\begin{align*}
&L_\ez((u^\ez-x_n)^2)\\
&=-H_{p_i}(Du^\ez)H_{p_j}(Du^\ez)[2(u^\ez_{x_j}-\delta_{nj})
(u^\ez_{x_i}-\delta_{ni})
+2(u^\ez-x_n)u^\ez_{x_ix_j}]\\
&\quad-4H_{p_i}(Du^\ez)H_{p_jp_l}(Du^\ez)u^\ez_{x_ix_j}(u^\ez-x_n)
(u^\ez_{x_l}-\delta_{nl}) -\ez[2(u^\ez_{x_i}-\delta_{ni})(u^\ez_{x_i}-\delta_{ni})
+2(u^\ez-x_n)u^\ez_{x_ix_i}]\\
&=-2\langle D_pH(Du^\ez),Du^\ez-e_n\rangle^2-4(u^\ez-x_n)\langle D^2_{pp}
H(Du^\ez)D[H(Du^\ez)],Du^\ez-e_n\rangle-2\ez|Du^\ez-e_n|^2,
\end{align*}
 and hence, by \eqref{cond1},
\begin{align*}
L_\ez((u^\ez-x_n)^2)\le -2\langle D_pH(Du^\ez),Du^\ez-e_n\rangle^2+4\tau|\langle D^2_{pp}H(Du^\ez)D[H(Du^\ez)],Du^\ez-e_n\rangle|.
\end{align*}
By Young's inequality,
\begin{align*}
&|\langle D^2_{pp}H(Du^\ez)D[H(Du^\ez)],Du^\ez-e_n\rangle|\\
&\le \frac12\langle D^2_{pp}H(Du^\ez)D[H(Du^\ez)], D[H(Du^\ez)]\rangle
+  \frac12\langle D^2_{pp}H(Du^\ez)(Du^\ez-e_n), (Du^\ez-e_n)\rangle.
\end{align*}
By the strongly concavity/convexity of $H$, we know that
$$\langle D^2_{pp}H(Du^\ez)(Du^\ez-e_n), (Du^\ez-e_n)\rangle\le \frac\Lambda2|Du^\ez-e_n|^2 $$
and
 $$\langle D_pH(Du^\ez),e_n-Du^\ez\rangle+\frac\lambda2|Du^\ez-e_n|^2\le H(e_n)-H(Du^\ez). $$
 Thus \begin{align*}
L_\ez((u^\ez-x_n)^2)&\le -2\langle D_pH(Du^\ez),Du^\ez-e_n\rangle^2+2 \tau
 \langle D^2_{pp}H(Du^\ez)D[H(Du^\ez)], D[H(Du^\ez)]\rangle\\
&\quad
+ 2 \tau \frac\Lambda\lambda [H(e_n)-H(Du^\ez)]-2 \tau \frac\Lambda\lambda\langle D_pH(Du^\ez),e_n-Du^\ez\rangle  ,\end{align*}

Pluging this in \eqref{EQ4.1}, one gets
\begin{align*}
&\int_{B(0,2)}\langle D_p H(Du^\ez),Du^\ez-e_n\rangle^2\Theta^\ez\,dx\\
&\le  \tau ^2+ 2\tau \int_{B(0,2)}
 \langle D^2_{pp}H(Du^\ez)D[H(Du^\ez)], D[H(Du^\ez)]\rangle\Theta^\ez\,dx\\
 & \quad+2 \tau \frac\Lambda\lambda\int_{B(0,2)} [H(e_n)-H(Du^\ez)]_+\Theta^\ez\,dx\\
 &\quad+  2 \tau\frac\Lambda\lambda \int_{B(0,2)}|\langle D_pH(Du^\ez),e_n-Du^\ez\rangle|  \Theta^\ez\,dx.
\end{align*}

Note that
 $$\int_{B(0,2)}
 \langle D^2_{pp}H(Du^\ez)D[H(Du^\ez)], D[H(Du^\ez)]\rangle\Theta^\ez\,dx\le C(\lambda,\Lambda).$$
By Young's inequality,
\begin{align*}
&2 \tau\frac\Lambda\lambda \int_{B(0,2)}|\langle D_pH(Du^\ez),e_n-Du^\ez\rangle|  \Theta^\ez\,dx\\
&\quad\le
\frac12 \int_{B(0,2)}\langle D_p H(Du^\ez),Du^\ez-e_n\rangle^2\Theta^\ez\,dx + \tau^2(\frac\Lambda\lambda)^2
\int_{B(0,2)}  \Theta^\ez\,dx\\
&\quad\le
\frac12 \int_{B(0,2)}\langle D_p H(Du^\ez),Du^\ez-e_n\rangle^2\Theta^\ez\,dx+ C(\lambda,\Lambda)\tau^2.
\end{align*}

\medskip
\noindent
{\it Step 3.}
 Set
$$v:=\zeta^2|Du^\ez-e_n|^2,$$
where $\zeta\in C^\fz_c(B(0,2))$, $0\le\zeta\le1$ in $B(0,2)$ and $\zeta=1$ in $B(0,1)$.
 Lemma \ref{L-L} gives that
\begin{equation}\label{FL1}
|Du^\ez(x_0)-e_n|^2=\int_{B(0,2)}L_\ez(\zeta^2|Du^\ez-e_n|^2)\Theta^\ez\,dx,
\end{equation}
where we used $\zeta|_{\partial B(0,2)}=0$.
One has
\begin{align}\label{FL2}
L_\ez(\zeta^2|Du^\ez-e_n|^2)&=|Du^\ez-e_n|^2L_\ez(\zeta^2)
+\zeta^2L_\ez(|Du^\ez-e_n|^2)\nonumber\\
&-2H_{p_i}(Du^\ez)H_{p_j}(Du^\ez)(\zeta^2)_{x_i}(|Du^\ez-e_n|^2)_{x_j}
-2\ez(\zeta^2)_{x_i}(|Du^\ez-e_n|^2)_{x_i}.\nonumber%\\
\end{align}

Owing to $L_\ez(u^\ez_s)=0$, we further compute
\begin{align*}
L_\ez(|Du^\ez-e_n|^2)&=-H_{p_i}(Du^\ez)H_{p_j}(Du^\ez)
[2u^\ez_{x_sx_j}u^\ez_{x_sx_i}+2(u^\ez_{x_s}-\delta_{ns})
u^\ez_{x_sx_ix_j}]\\
&\quad-4(u^\ez_s-\delta_{ns}) H_{p_i}(Du^\ez)H_{p_jp_l}(Du^\ez)u^\ez_{x_ix_j}u^\ez_{x_sx_l}\\
 &\quad-\ez [2u^\ez_{x_sx_i}u^\ez_{x_sx_i}+2(u^\ez_s-\delta_{ns})
u^\ez_{x_sx_ix_i}]\\
&=-2|D[H(Du^\ez)]|^2-2\ez|D^2u^\ez|^2.
\end{align*}

Note that
\begin{align*}
L_\ez(\zeta^2)&=-H_{p_i}(Du^\ez)H_{p_j}(Du^\ez)[2\zeta_{x_i}\zeta_{x_j}+2\zeta\zeta_{x_ix_j}]-4\zeta H_{p_i}(Du^\ez)H_{p_jp_l}(Du^\ez)u^\ez_{x_ix_j}\zeta_{x_l}\\
 &\quad -\ez[2\zeta_{x_i}\zeta_{x_i}+2\zeta\zeta_{x_ix_i}]\\
 &\le C|D_pH(Du^\ez)|^2+C\Lambda \zeta|D[H(Du^\ez)]|
 +C\ez
\end{align*}
and hence
$$|Du^\ez-e_n|^2L_\ez(\zeta^2)\le \frac18 \zeta^2|D[H(Du^\ez)]|^2+ C(\lambda,\Lambda)|Du^\ez-e_n|^2[1+ | Du^\ez |^2 ]$$

and
\begin{align*}
&-2H_{p_i}(Du^\ez)H_{p_j}(Du^\ez)(\zeta^2)_{x_i}(|Du^\ez-e_n|^2)_{x_j}-2\ez(\zeta^2)_{x_i}(|Du^\ez-e_n|^2)_{x_i}\\
&\quad\le C|D[H(Du)]||D_pH(Du^\ez)||Du^\ez-e_n|\zeta+C\zeta\ez
 |Du^\ez-e_n||D^2u^\ez|\\
 &\quad\le \frac 18 \zeta^2[|D[H(Du^\ez)]|^2+\ez |D^2u^\ez|^2]+ C(\lambda,\Lambda)|Du^\ez-e_n|^2[1+ | Du^\ez |^2 ] .\end{align*}

%Since Theorem \ref{ex-un} implies that $|D_pH(Du^\ez)|\le \Lambda |Du^\ez|$ in $B(0,2)$,
%
We conclude that
\begin{align*}
L_\ez(\zeta^2|Du^\ez-e_n|^2) \le C(\lambda,\Lambda)|Du^\ez-e_n|^2[1+ | Du^\ez |^2 ].
\end{align*}
In view of \eqref{FL1}, we conclude that
\begin{equation}\label{FL5}
|Du^\ez(x_0)^2-e_n|^2\le C(\lambda,\Lambda)\|Du^\ez\|_{L^\fz(V)}
\int_{B(0,2)}|Du^\ez-e_n|^2\Theta^\ez\,dx.
\end{equation}
Since $H$ is strongly convex,
\begin{equation*}
H(e_n)\ge H(Du^\ez)+\langle D_pH(Du^\ez),e_n-Du^\ez\rangle+
\frac{\lambda |Du^\ez-e_n|^2}{2}.
\end{equation*}
This implies that
\begin{align*}
\frac{\lambda }{2}\int_{B(0,2)}|Du^\ez-e_n|^2\Theta^\ez\,dx
&\le -\int_{B(0,2)}\langle D_pH(Du^\ez),Du^\ez-e_n\rangle\Theta^\ez\,dx +\int_{B(0,2)}[H(e_n)-H(Du^\ez)]_+\Theta^\ez\,dx\\
&\le  C(\lambda,\Lambda)[\tau+\dz+\frac1\ez  e^{ -\frac{\mu}{\ez}\dz  }].
\end{align*}
The proof of Theorem \ref{Fla} is complete.

\bigskip

 \noindent {\bf Acknowledgment.}
The authors
 would like to thank the supports of
 National Natural Science of Foundation of China (No. 11522102\&11871088).

\end{document}